\documentclass[10pt,reqno]{amsart}

\newcommand{\olsi}[1]{\,\overline{\!{#1}}} 
\newcommand{\ols}[1]{\mskip.5\thinmuskip\overline{\mskip-.5\thinmuskip {#1} \mskip-.5\thinmuskip}\mskip.5\thinmuskip} 

\usepackage{amsmath,amsfonts,amssymb,amsthm,mathrsfs}
\usepackage{hyperref}
\usepackage{bm}
\usepackage{color}
\usepackage{esint}
\usepackage{graphicx}
\usepackage{graphics}
\usepackage{geometry}
\usepackage{xcolor}
\usepackage[shortlabels]{enumitem}
\usepackage{physics}
\usepackage{relsize}
\usepackage{MnSymbol}
\usepackage[all]{xy}

\newtheorem{theorem}{Theorem}[section]

\newtheorem{lemma}[theorem]{Lemma}
\newtheorem{corollary}[theorem]{Corollary}
\newtheorem{proposition}[theorem]{Proposition}
\newtheorem{conjecture}[theorem]{Conjecture}
\newtheorem{remark}[theorem]{Remark}
\newtheorem*{claim}{Claim}
\theoremstyle{definition}
\newtheorem{definition}[theorem]{Definition}
\newtheorem{question}[theorem]{Question}
\newtheorem{example}[theorem]{Example}

\numberwithin{equation}{section}

\newcommand{\pp}{\partial}

\newcommand{\RR}{\mathbb{R}}

\newcommand{\bp}{\bar\partial}

\newcommand{\Ca}{\mathcal{C}}
\newcommand{\dP}{\mathbb{P}}

\begin{document}

\title[\resizebox{5.5in}{!}{Calabi-Yau metrics of Calabi type with polynomial rate of convergence}]{Calabi-Yau metrics of Calabi type with polynomial rate of convergence}
\author{Yifan Chen}
\address{Department of Mathematics\\
University of California\\
Berkeley, CA, USA, 94720\\}
\email{yifan-chen@berkeley.edu}
\thanks{}

\begin{abstract}

We present new complete Calabi-Yau metrics defined on the complement of a smooth anticanonical divisor with ample normal bundle, approaching the Calabi model space at a polynomial rate. Moreover, we establish the uniqueness of this type of Calabi-Yau metric within a fixed cohomology class.

\end{abstract}
\maketitle
\section{Introduction}

In 1978, Yau \cite{yau1978ricci} gave the celebrated proof of Calabi conjecture by solving the Monge-Amp\`ere equation for each K\"ahler class on compact K\"ahler manifolds. In 1990, Tian and Yau \cite{tian1990complete} expanded upon this achievement by constructing complete Calabi-Yau metrics on quasi-projective manifolds, extending the techniques introduced in \cite{yau1978ricci} to the non-compact case. Specifically, when $M$ is a Fano manifold and $D$ is a smooth irreducible anticanonical divisor, they constructed complete Calabi-Yau metric $\omega_{TY}$ on $M\setminus D$, called Tian-Yau metric. $\omega_{TY}$ is exponentially close to the Calabi model space. Here the Calabi model space $(\mathcal{C}, I_\Ca, \omega_\Ca, \Omega_\Ca)$ is the disc bundle over $D$ within $N_D$ removing the zero section, with complex structure $I_\Ca$ given by $N_D$, K\"ahler metric $\omega_{\Ca}$ given by Calabi anstaz and a nowhere vanishing $(n,0)$-form $\Omega_\Ca$. The strict definition of Calabi model space will be given in Section \ref{modelspace}. In fact, for any compact supported class in $M\setminus D$, we can find Calabi-Yau metric exponentially close to the Calabi model space. We call the complete Calabi-Yau manifold with this property \emph{asymptotically Calabi}, which is defined as follows. 

\begin{definition}\label{asymptotic calabi}
Let $X$ be an $n$-dim complete K\"ahler manifold with, complex structure $I$, 2-form $\omega$ and $(n,0)$-form $\Omega$. We say $(X, I, \omega, \Omega)$ is 
\begin{enumerate}
\item \emph{weak asymptotically Calabi with rate $\kappa$} if there exists $\underline{\delta}>0$, $\kappa>0$, a Calabi model space $(\mathcal{C}, I_\Ca, \omega_\Ca,\Omega_\Ca)$ with function $z$ on $\Ca$ defined in Section \ref{modelspace}, and a diffeomorphism $\Phi: \mathcal{C} \setminus \mathcal{K} \rightarrow X \setminus K$, where $K \subset X$ and $\mathcal{K}\subset \mathcal{C}$ are compact, such that the following hold uniformly as $z \rightarrow+\infty$:
	$$\left|\nabla_{\omega_\Ca}^k\left(\Phi^* I-I_{\mathcal{C}}\right)\right|_{\omega_\Ca}+\left|\nabla_{\omega_\Ca}^k\left(\Phi^* \Omega-\Omega_{\mathcal{C}}\right)\right|_{\omega_\Ca} =O(e^{-\underline{\delta} z^{n / 2}}),$$
	$$\quad\left|\nabla_{\omega_\Ca}^k\left(\Phi^* \omega-\omega_{\mathcal{C}}\right)\right|_{\omega_\Ca} =O(z^{-\kappa}) \text{ for all } k \in \mathbb{N}_0.$$
\item \emph{asymptotically Calabi} if it is weak asymptotically Calabi and 
	$$\quad\left|\nabla_{\omega_\Ca}^k\left(\Phi^* \omega-\omega_{\mathcal{C}}\right)\right|_{\omega_\Ca}=O(e^{-\underline{\delta} z^{n / 2}})\text{ for all } k \in \mathbb{N}_0.$$ 
\end{enumerate}
\end{definition}

The asymptotically Calabi Calabi-Yau manifold is well understood by Hein-Sun-Viaclovsky-Zhang \cite{hsvz2}, in which they showed that any asymptotically Calabi Calabi-Yau manifold can be compactified complex analytically to a weak Fano manifold, i.e. a smooth projective manifold with nef and big anti-canonical bundle. In this paper, we are going to show the existence and uniqueness of weak asymptotically Calabi Calabi-Yau metrics that is not asymptotically Calabi. Our setting is as follows:
\begin{definition}\label{setting}
	Let $M$ be a compact K\"ahler manifold with complex dimension $n\geqslant 3$, $D\in |-K_M|$ be a smooth divisor with ample normal bundle and $X = M\setminus D$. We denote $H_{+}^2(X)$ as the subset of $H^2(X)$ which consists of classes $\mathfrak{k}$ such that $\mathfrak{k}^p$ is positively paired with any compact analytic subset $Y$ of $X$ of pure complex dimension $p$. 
\end{definition}

 The main result of this paper is the following

\begin{theorem}\label{main theorem 1}
	 For any class $\mathfrak{k}$ in $H_{+}^2(X)$, there exists a Calabi-Yau metric $\omega$ in the class $\mathfrak{k}$ which is weak asymptotically Calabi with rate $1$.
\end{theorem}

To show the existence, we use the method in Tian-Yau \cite{tian1990complete}, which was subsequently generalized and refined by Hein \cite{HeinThesis}. This Tian-Yau-Hein's package facilitates the production of complete Calabi-Yau metrics when a suitable model metric at infinity is known. However, we cannot directly apply it here because the 2-form on $X$ coming from the restriction of 2-forms on $M$ will only decay at the rate $r^{-\frac{2}{n+1}}$, while Tian-Yau-Hein's package requires the Ricci potential decays faster than $r^{-2}$. We need to modify our background metric by solving the linearized operator of Monge-Amp\`ere equation to improve the decay. 

\begin{remark}
	The metric in Theorem \ref{main theorem 1} is new in the sense that under a fixed diffeomorphism $\Phi$, $\omega$ is weak asymptotically Calabi with rate $1$ but not asymptotically Calabi.
	We can also find Calabi-Yau metrics when $\mathrm{dim}_{\mathbb{C}}M = 2$. However, these metrics are not new. In fact, the 2 dimensional case is well understood: Sun-Zhang \cite{sun2022collapsing} showed that the $ALH^*$ gravitational instanton is always asymptotically Calabi.
\end{remark}

Besides, we also prove that the constructed metric is unique in the sense that:
\begin{theorem}\label{unique theorem 1}
     Fix any $p\in X$. Let $r$ be the distance function towards $p$ under the metric $\omega$ in Theorem \ref{main theorem 1}. If we have another Calabi-Yau metric $\tilde{\omega}$ in the same class $\mathfrak{k}$ satisfying $$\left|\,\tilde{\omega}-\omega\right|_{\omega}\leqslant C\cdot r^{-\kappa},\text{ as } r\to \infty,$$ for some positive constant $C,\kappa$, then $\tilde{\omega} = \omega$.
\end{theorem}

One could always have uniqueness if the metrics in Theorem \ref{unique theorem 1} are close with rate $O(r^{-N})$ for some $N$ large enough. Nevertheless, this rate is much faster than the decay rate of our metric constructed in Theorem \ref{main theorem 1}. The proof of uniqueness theorem under any polynomial closeness requires two ingredients: the $i\pp\bp$ lemma with the $L^2$ estimate on $X$, and the Liouville type theorem on $\Ca$.

\subsection*{Outline of the paper}
The paper is organized as follows. In Section \ref{modelspace}, we provide an introductory overview of the Calabi ansatz, denoted by $\omega_{\Ca}$, within an open neighborhood $\mathcal{C}$ of the divisor $D$ in its normal bundle, excluding the zero section. Our exposition primarily adheres to the notation and discussion presented by Hein, Sun, Viaclovsky, and Zhang \cite{hsvz1}. Furthermore, we briefly review the analytical framework developed by Tian, Yau, and Hein \cite{HeinThesis}, list some geometric properties of the Calabi model space that facilitates the application of uniform elliptic estimates later.

In section \ref{Poisson equation}, we have the solution of the Poisson equation with appropriate weighted estimates. The method of variable separation, as detailed by Sun and Zhang \cite{SZ}, allows us to simplify the Poisson equation in the model space into a particular form of ordinary differential equation. Leveraging the solutions' estimates for these ordinary differential equations, as presented in Appendix \ref{estimate of ODE}, we construct an inversion of the Laplacian operator within suitably weighted spaces. This enables us to initiate the iterative processes detailed in Sections \ref{decay faster than quadratic} and \ref{integral condition}.

In Sections \ref{decay faster than quadratic} and \ref{integral condition}, we construct a good background metric $\omega$ within the cohomology class $\mathfrak{k}$. We solve the Poisson equation on the model space iteratively to enhance the decay rate of the Ricci potential of $\omega$, as detailed in Section \ref{decay faster than quadratic}. In Section \ref{integral condition}, we refine the K\"ahler potential by incorporating the harmonic moment map $z$ alongside other pluri-subharmonic functions. This adjustment ensures not only the positivity of $\omega$ but also its compliance with the integral condition outlined in Tian-Yau-Hein's package \cite{HeinThesis}. The iterative method adopted here is inspired by Conlon and Hein \cite{conlon2013asymptotically}. The specific technique of modifying the potential via the harmonic moment map $z$ is equivalent to choosing appropriate scaling of the metric $h_D$ on the normal bundle in Hein, Sun, Viaclovsky, and Zhang \cite{hsvz2}.

In Section \ref{behavior of beta and main theorem}, we deform our good background metric $\omega$ to a genuine Calabi-Yau metric on $X$ by Tian-Yau-Hein's package. We also show that the perturbed metric is weak asymptotically Calabi with rate $1$. This requires a slight generalization of Hein's decay result in \cite{HeinThesis}.

In Section \ref{Uniqueness}, we discuss the uniqueness with restricted asymptotics of the Calabi-Yau metric in the fixed class. We first prove the $i\pp\bp$-lemma on $X$ based on the global H\"ormander $L^2$ estimate. Then by our solution of the Poisson equation and the behavior of the harmonic function on the model space we can deduce the global $C^0$ estimate to do integration by parts.

In Section \ref{discussions}, we present some examples of Calabi-Yau manifolds which are weak asymptotically Calabi but not asymptotically Calabi under a fixed diffeomorphism. We also make some conjectures about the stronger uniqueness theorem and the compactification and classification of weak asymptotically Calabi manifold or under even weaker condition. 

In the following sections, $C$ and $\epsilon$ will be two uniform constants that may vary.

\subsection*{Other works on complete non-compact Calabi-Yau manifolds}
There are many progress in the exploration of new complete Calabi-Yau manifolds, extending the seminal work of Tian and Yau to cases where $D$ is singular. Collins-Li \cite{collins2022complete} constructed new complete Calabi-Yau metrics when $D$ consists of two proportional transversely intersecting smooth divisors. The metric in their construction is $i\pp\bp$ exact and polynomially closed to the generalized Calabi anstaz. Later Collins-Tong-Yau \cite{collins2024free} solved a certain free boundary Monge-Amp\`ere equation crucial to the inductive strategy proposed in \cite{collins2022complete} to deal with the general case when $D$ is simple normal crossing. 

There are also many interesting works on this problem based on Tian-Yau-Hein's package. For example, the non-flat Calabi-Yau metric on $\mathbb{C}^n$ constructed by Li \cite{li2019new}, Sz{\'e}kelyhidi \cite{szekelyhidi2020new} and Conlon-Rochon \cite{conlon2017new} with maximal volume growth, and Min \cite{min2023construction} with volume growth $2n-1$ when $n$ is even. Those works constructed Calabi-Yau with singular tangent cone at infinity. 

Recently, Apostolov-Cifarelli \cite{apostolov2023hamiltonian} constructed new complete Calabi-Yau metrics on $\mathbb{C}^n$ with volume growth $2n-1$. Their new method using toric geometry and Hamiltonian 2-forms is significantly different from Tian-Yau-Hein package and produces exotic complete Calabi-Yau metrics with interesting behavior at infinity.

\subsection*{Acknowledgement}

The author is deeply grateful to Professor Song Sun for suggesting this problem, enlightening discussion and constant support. The author thanks Junsheng Zhang for fruitful conversation and encouragement, and thanks Yueqing Feng and Hongyi Liu for reading the draft of the paper and many helpful comments.The author also thanks IASM for their hospitality during the visit when the research was partially carried out, and the NSF for the generosity of the grant DMS-2304692.

\section{Calabi model space}\label{modelspace}

Let us give a brief introduction of Calabi ansatz and some notations we will use later. The notations in this section mainly follow Hein-Sun-Viaclovsky-Zhang \cite[Section 3]{hsvz1}.
\subsection{Calabi ansatz}
Let $M$ be an $n$-dimensional compact K\"ahler manifold with $n\geqslant 2$, $D$ be a smooth anticanonical divisor with ample normal bundle $N_D$ and let $X=M\setminus D$ denote the complement of $D$ in $M$. By adjunction formula we know that $D$ has trivial canonical bundle with a nowhere-vanishing holomorphic volume form $\Omega_D$ such that $$(2\pi c_1(N_D))^{n-1}=\int_Di^{(n-1)^2}\Omega_D\wedge\ols\Omega_D.$$ Hence by Yau's theorem, up to scaling we have a unique hermitian metric $h_D$ on $N_D=\left.-K_{M}\right|_D$ with curvature form $\omega_D$ on $D$ in the class $2\pi c_1(D)$ satisfying 
$$\omega_D^{n-1}=i^{(n-1)^2}\Omega_D\wedge\ols\Omega_D.$$
We fix such an $h_D$. For any point $\xi\in N_D$, let $t = -\log \vert \xi \vert^2_{h_D}$ and $z = t^{\frac{1}{n}}$ be functions on the complement of the zero section in the total space $N_D \setminus D$. 
Let $\mathcal{C}=\{0<\left|\xi\right|_{h_D}<1\}\subseteq N_D$ be the disc bundle over $D$ with complex structure $I_\Ca$ restricted from $N_D$. On $\mathcal{C}$, we have a Calabi-Yau metric given by the Calabi ansatz:
$$\omega_{\mathcal{C}} = \frac{n}{n+1}i\pp\bp (t^{\frac{n+1}{n}})= zi\pp\bp t +\frac{1}{nz^{n-1}}i\pp t\wedge\bp t.$$ 
Let $\Omega_{\mathcal{C}}$ denotes the unique holomorphic $(n,0)$-form on $N_D$ such that $$Z\righthalfcup \Omega_{\mathcal{C}} = \pi^* \Omega_D,$$ where $Z$ denotes the holomorphic vector field generated by the  scalar multiplication along the fiber direction and $\pi:N_D\to D$ is the projection map. And the data $(\Ca, I_\Ca, \omega_\Ca, \Omega_\Ca)$ is called \textit{Calabi model space}.

On $M$ we can also choose a holomorphic section $S$ of $-K_M$ to be the defining function of $D$ such that $S^{-1}$ can be seen as an $(n,0)$-form $\Omega_X$ on $X$ with a simple pole and residue $\Omega_D$ along $D$. We choose a metric $h_M$ of $-K_M$ such that $ \left.h_M\right|_D = h_D $. Then we can construct the following $(1,1)$-form on $X$:
\begin{align*}
    \omega_X = \frac{n}{n+1} i\pp \bp\left(-\log \left|S\right|_{h_M}^2\right)^{\tfrac{n+1}{n}}.
\end{align*}
As proved in Hein-Sun-Viaclovsky-Zhang \cite[Proposition 3.4.]{hsvz1}, it is asymptotically Calabi in the following way:
\begin{proposition}\label{exponential close}
	The complex structure $I_X$ and $I_C$, metric $\omega_\mathcal{C}$ and $\omega_{X}$ and canonical form $\Omega_X$ and $\Omega_{\mathcal{C}}$ are exponentially closed. To be more precise: there exists a compact set $K$ in $X$, a compact set $\mathcal{K}$ in $\Ca$,  and a diffeomorphism induced by exponential map $\Phi: \mathcal{C}\setminus \mathcal{K}\to X\setminus K$ such that for all $k \geqslant 0, \epsilon>0$:
$$\left|\nabla_{g_{\mathcal{C}}}^k\left(\Phi^* I_X-I_{\mathcal{C}}\right)\right|_{g_{\mathcal{C}}}+\left|\nabla_{g_{\mathcal{C}}}^k\left(\Phi^* \Omega_X-\Omega_{\mathcal{C}}\right)\right|_{g_{\mathcal{C}}}+\left|\nabla_{g_{\mathcal{C}}}^k\left(\Phi^* \omega_X-\omega_{\mathcal{C}}\right)\right|_{g_{\mathcal{C}}}=O(e^{-(\frac{1}{2}-\epsilon) z^n}).$$
\end{proposition}


\subsection{Geometric properties and Tian-Yau-Hein's package}

Let us first list some geometry of Calabi model space directly coming from the formula of $\omega_\Ca$. $\omega_\Ca$ is complete when $\left|\xi\right|_{h_D} \to 0$ and incomplete when $ \left|\xi\right|_{h_D} \to 1 $. It has volume growth of order $\tfrac{2n}{n+1}$ and its sectional curvature decays at the rate $r^{-\frac{2}{n+1}}$. Any distance function on $\Ca$ will be comparable to $z^{\frac{n+1}{2}}$.

Then we introduce some requirements of the base Riemannian manifold $(N, g)$ in Tian-Yau-Hein's package and show that our model space $(\Ca, \omega_\Ca)$ satisfies those properties. One may refer to \cite{HeinThesis} for details and examples.

We begin with the definition of $\mathrm{SOB}(\nu)$ property:
\begin{definition}
    Let $(N,g)$ be a complete noncompact Riemannian manifold with real dimension at least 3. We say $(N,g)$ satisfy $\mathrm{SOB}(\nu)$ condition for some $\nu>0$ if and only if there exists a point $p\in N$ and a positive constant $C>1$ such that
    \begin{enumerate}
        \item Volume growth is at most $\nu$, i.e. $\mathrm{Vol}_g(B(p,R))\leqslant CR^\nu$ for all $R>C$.
        \item $\mathrm{Ric}(x)\geqslant -Cd_g(x,p)^{-2}$, $\forall x \in N$.
        \item $\mathrm{Vol}_g(B(x,(1-\tfrac{1}{C})d_g(x,p)))\geqslant \tfrac{1}{C}d_g(x,p)^\nu$.
        \item For any $D>C$, any two points $x,y\in N$ with $ d_g(p,x)= d_g(p,y) = D$ can be joined by a curve of length at most $C\cdot D$, lying in the annulus $A(p,\tfrac{1}{C}D,CD):=\{x\in N|\tfrac{1}{C}D<d_g(p,x)<CD\}$.
    \end{enumerate}
\end{definition}
\begin{remark}
    In the original definition of $\mathrm{SOB}(\nu)$ property in \cite{hein2012gravitational}, we need that the annulus $A(p,s,t)$ is connected for any $t>s>0$. To apply Theorem \ref{Hein's Package} it suffices to check the RCA property (4) instead of the connectivity of the annulus.
\end{remark}

\begin{proposition}\label{model space SOB}
	$(\Ca,\omega_\Ca)$ has $\mathrm{SOB}(\tfrac{2n}{n+1})$ property.
\end{proposition}
\begin{proof}
	$(2)$ follows from the Ricci-flat property of $\Ca$. For $(1)$, recall that $\omega_\mathcal{C} = zi\pp\bp t+\frac{1}{nz^{n-1}}i\pp t \wedge\bp t$ and the volume form $\omega_{\Ca}^n = (i\pp\bp t)^{n-1}\wedge i\pp t \wedge\bp t$. The distance function $r_{\omega_\Ca}$ to some fixed point $p$ in $\mathcal{C}$ is comparable to $z^{\frac{n+1}{2}}$. Consequently, one can see that $\tilde{A}(z_1,z_2) = \{x\in \mathcal{C}\;|\;z_1<z(x)<z_2\}$ is comparable to the annulus $A(p,R_1,R_2)$ in the sense that: for any $z_2>z_1>C$ and $R_2>R_1>C$, we have
\begin{align*}
   \tilde{A}(z_1,z_2)&\subseteq A\left(p,\tfrac{1}{C}z_1^{\frac{n+1}{2}},Cz_2^{\frac{n+1}{2}}\right),\\
   A(p,R_1,R_2)&\subseteq \tilde{A}\left(\left(\tfrac{1}{C}R_1\right)^{\frac{2}{n+1}},\left(CR_2\right)^{\frac{2}{n+1}}\right).
\end{align*}
Similarly, for any $R\leqslant C$, we have:
\begin{align*}
    B(p,R)\subseteq \mathcal{K}\cup \tilde{A}\left(0, (CR)^{\frac{2}{n+1}}\right).
\end{align*}
Also, we can see from the ansatz that the diameter of $\{x\in \Ca\,|\,z = z_0\}$  is comparable to $\sqrt{z_0}$, which shows that 
\begin{align*}
    \tilde{A}\left(z(x)-\tfrac{1}{C}\left(z(x)-z(p)\right),z(x)+\tfrac{1}{C}\left(z(x)-z(p)\right)\right)\subset B\left(x,\left(1-\tfrac{1}{C}\right)d_g(x,p)\right).
\end{align*}
With all these equivalence, we know that
\begin{align*}
    \mathrm{Vol}_g(B(p,R)) \leqslant C+ \int_{\left\{-(CR)^{\frac{2n}{n+1}}\leqslant \,t\, \leqslant (CR)^{\frac{2n}{n+1}}\right\}} (i\pp\bp t)^{n-1}\wedge dt\wedge d^ct  \leqslant  CR^{\frac{2n}{n+1}}.
\end{align*}
Similarly for $(3)$ we have 
\begin{align*}
	\mathrm{Vol}_g(B\left(x,\left(1-\tfrac{1}{C}\right)d_g(x,p)\right)) \geqslant C\cdot z(x)^{n}\geqslant C\cdot d_g(x,p)^{\frac{2n}{n+1}}.
\end{align*}
To show $(4)$, for $D$ large enough and any two points $x$ and $y$ with $d_g(x,p)= d_g(y,p) = D$, we have $z(x),z(y) \in \left(\tfrac{1}{C}\cdot D^{\frac{2}{n+1}}, C\cdot D^{\frac{2}{n+1}}\right)$. Then by the formula of $\omega_\mathcal{C}$ we can join $x$ and $y$ by a curve of length at most $C\cdot \left( D+D^{\frac{1}{n+1}}\right)$ lying in $A\left(\tfrac{1}{C}\cdot D^{\frac{2}{n+1}}, C\cdot D^{\frac{2}{n+1}}\right)$ and consequently in the annulus $A(p,\tfrac{1}{C}D,CD)$. 

Hence we have the $\mathrm{SOB}(\tfrac{2n}{n+1})$ condition on $(\mathcal{C},\omega_\mathcal{C})$. 
\end{proof}

We continue with the definition of $\operatorname{HMG}(\lambda,k,\alpha)$ property:

\begin{definition}
	We say that $\left(N^n, g\right)$ is $\operatorname{HMG}(\lambda, k, \alpha)$, for some $\lambda \in[0,1]$, $k \in \mathbb{N}_0$, $\alpha \in(0,1)$, if there exist $x_0 \in N$ and $C \geqslant 1$ such that 
	\begin{enumerate}
		\item for every $x \in N$ with $r(x) \geqslant C$ there exists a local holomorphic diffeomorphism $\Phi_x$ from the unit ball $B \subset \mathbb{R}^n$ into $N$ such that $\Phi_x(0)=x$ and $\Phi_x(B) \supset B\left(x, \tfrac{1}{C} r(x)^\lambda\right)$, 
		\item $h:=r(x)^{-2 \lambda} \Phi_x^* g$ satisfies $\operatorname{Inj}(h) \geqslant \tfrac{1}{C}, \tfrac{1}{C} g_{\text {euc}} \leqslant h \leqslant C g_{\text {euc}}$, and $\left\|h-g_{\text {euc}}\right\|_{C^{k, \alpha}\left(B, g_{\text{euc}}\right)} \leqslant C$.
	\end{enumerate}
\end{definition}

Tian-Yau \cite[Proposition 1.2.]{tian1990complete} provides a simple criterion for a complete K\"ahler manifold $(N,\omega)$ to be $\operatorname{HMG}(0,k,\alpha)$. We refer to Hein's thesis \cite[Lemma 4.7.]{HeinThesis} for a slightly generalized statement and sketch of the proof.
\begin{lemma}\label{lemma of C^k,alpha}
	A complete Kähler manifold with $\left|\mathrm{Rm}\right|+\sum_{i=1}^k r^{i \lambda} \left| \nabla^i \mathrm{Scal} \right| \leqslant C r^{-2 \lambda}$ for some $k \in \mathbb{N}_0$ and $\lambda \in[0,1]$ is $\operatorname{HMG}(\lambda, k+1, \alpha)$ for every $\alpha \in(0,1)$.
\end{lemma}

\begin{proposition}
	$(\Ca,\omega_\Ca)$ has $\mathrm{HMG}(\tfrac{1}{n+1},k,\alpha)$ property, for any $k\in \mathbb{N}_0$ and $\alpha \in(0,1)$.
\end{proposition}
\begin{proof}
	The proof goes almost verbatim with the proof of Lemma \ref{lemma of C^k,alpha}. The proof of Lemma \ref{lemma of C^k,alpha} only used the completeness to guarantee that the injectivity radius of local universal cover around $x_0$ has uniform lower bound independent of $x_0$. Since $\Ca$ is the disc bundle over $D$, after we do rescaling by $\tilde\omega_\Ca = z(x_0)^{-1}\omega_\Ca$, the $S^1$ action gives the only collapsing direction which disappears after passing to the local universal cover. So the local universal cover has uniform curvature bound and is volume non-collapsing, which leads to uniform injectivity radius lower bound.
	\end{proof}

Now we are ready to present the following result taken from Hein's thesis \cite{HeinThesis} which is a powerful tool to give the existence of Calabi-Yau metric on complete noncompact K\"ahler manifold.
\begin{theorem}[Tian-Yau-Hein's Package]\label{Hein's Package}
	Let $(X^n, \omega)$ be a complete noncompact K\"ahler manifold, which satisfies the condition $\mathrm{SOB}(\nu)$ and $\mathrm{HMG}(0,3,\alpha)$ for some $\nu>0$, $0<\alpha<1$. Let $r$ be the distance function to a fixed point $p \in X$ with respect to the metric $\omega$. Let $f \in C^{2, \alpha}(X)$ satisfy $|f| \leqslant C r^{-\mu}$ on $\{r>1\}$ for some $\mu>2$ and $\int_X\left(e^f-1\right) \omega^n=0$. Then there exist $\bar{\alpha} \in(0, \alpha]$ and $u \in C^{4, \bar{\alpha}}(X)$ such that $\left(\omega+i \partial \bar{\partial} u\right)^n=e^f \omega^n$. Moreover $\int_X|\nabla u|^2 \omega^n<\infty$. If in addition $f \in C_{\mathrm{loc}}^{k, \bar{\alpha}}(X)$ for some $k \geqslant 3$, then all such solutions $u$ belong to $C_{\mathrm{loc}}^{k+2, \bar{\alpha}}(X)$.
\end{theorem}

\begin{remark}
	We can not directly use Theorem \ref{Hein's Package} because in our setting, the decay rate of the class $\mathfrak{k}$ is only $r^{-\frac{2}{n+1}}$ which is slower than $r^{-2}$. So we need to modify the representative $\beta$ by $i\pp\bp u$ for some function $u$ on $X$, which comes from the suitable solution of Poisson equation on model space $\Ca$, which we will discuss in section \ref{Poisson equation}.
\end{remark}

\subsection{Uniform elliptic estimates}
The previous properties are mainly used to guarantee that we have weighted Sobolev inequality, weighted H\"older space, and can do weighted elliptic estimates. 

Recall that the $C^{k,\alpha}$ norm of a function $f$ on a ball $B_g(x,r)$ inside a manifold $(X,g)$ is 
\begin{align*}
    \left\|f\right\|_{C^{k,\alpha}(B_g(x,r))} = \sum_{j=1}^{k-1}\left\|\nabla_g^j f\right\|_{C^0(B_g(x,r))} + \left\|\nabla_g^k f\right\|_{C^{0,\alpha}(B_g(x,r))}.
\end{align*}

If we scale the metric by a constant $\frac{1}{z_0}$, i.e. $\tilde{g} = \frac{1}{z_0} g$, we then have 
\begin{align*}
    &\left|\nabla_{\tilde{g}}^j f\right|_{\tilde{g}} = z_0^{\frac{j}{2}}\cdot |\nabla_g^j f|_g,\\
    &\left\|f\right\|_{W^{k,p}(B_{\tilde{g}}(x,r))} = \sum_{j=0}^kz_0^{\frac{j}{2}}\left\|\nabla_g^j f\right\|_{L^p(B_g(x,\sqrt{z_0}r))},\\
    &\left\|f\right\|_{C^{k,\alpha}(B_{\tilde{g}}(x,r))} = \sum_{j=0}^{k-1}z_0^{\frac{j}{2}}\left\|\nabla_g^j f\right\|_{C^0(B_g(x,\sqrt{z_0}r))} + z_0^{\frac{k+\alpha}{2}}\left\|\nabla_g^k f\right\|_{C^{0,\alpha}(B_g(x,\sqrt{z_0}r))}.
\end{align*}
Now we can prove the uniform Schauder estimate here for future reference:


\begin{proposition}\label{uniform schauder estimates}
	Let $(N,g)$ be a manifold satisfying $\mathrm{HMG}(\lambda,k,\alpha)$ property. Let $u$ and $v$ be smooth functions on $M$ such that $\Delta_g u = v$. Then there exist some constants $C_k$ such that
	\begin{align*}
r(x)^{k\lambda}\left\|\nabla^k u\right\|_{C^0(B_g(x,r(x)^\lambda))}\leqslant C_{k} \left(\left\| u\right\|_{C^0(B_g(x,r(x)^\lambda))}+\sum_{i=0}^{k-1}r(x)^{(i+2)\lambda}\left\|\nabla^i v\right\|_{C^0(B_g(x,r(x)^\lambda))}\right).
\end{align*}
\end{proposition}
\begin{proof}
	Fix any $x$ in $N$. Let $h = r(x)^{-2\lambda}\Phi^*g$ be the rescaled pull back metric on the unit ball $B(0,1)$ in $\mathbb{R}^n$, let $\tilde u = \Phi_x^* u$, $\tilde v = r(x)^{2\lambda}\Phi_x^* v$. Since we use the pull back metric, the Laplacian is preserved. We have $\Delta_h \tilde u = \tilde v$.
	Write this elliptic equation under the Euclidean coordinate, with $\|\nabla^k h\|_{g_{\text{euc}}}\leqslant C(k)$ for any integer $k\geqslant 0$, we have $\tilde u$ satisfies the following elliptic equation 
	\begin{align}\label{poisson equation euclidean}
    \frac{1}{\sqrt{\operatorname{det} h}} \frac{\partial}{\partial x^i}\left(h^{ij}\sqrt{\operatorname{det} h}\frac{\partial}{\partial x^j}\tilde u\right) = \tilde v.
	\end{align}
	By the standard elliptic estimates on the Euclidean space and passing to the original metric $g$ we get the required estimate. To be more precise, by $W^{2,q}$ estimates we know that there exists a uniform constant $C_q$ only depends on $q$ such that for any $1< q <+\infty$
\begin{align*}
    \left\|\tilde u\right\|_{W^{2,q}\left(B(0,\frac{1}{2})\right)}\leqslant C_q\cdot \left(\left\|\tilde v\right\|_{C^0(B(0,1))}+\left\|\tilde u\right\|_{C^0(B(0,1))}\right).
\end{align*}

Consequently by Sobolev embedding, we know 
\begin{align*}
    \left\|\tilde u\right\|_{C^{1,1-\frac{2n}{q}}\left(B(0,\frac{1}{2})\right)}\leqslant C_q\left\|\tilde u\right\|_{W^{2,q}\left(B(0,\frac{1}{2})\right)}\leqslant C_q\left(\left\|\tilde v\right\|_{C^0(B(0,1))}+\left\|\tilde u\right\|_{C^0(B(0,1))}\right).
\end{align*}

Taking derivative of (\ref{poisson equation euclidean}) under the Euclidean coordinate, we have:
\begin{align*}
    \left\|\tilde u\right\|_{W^{3,q}\left(B(0,\frac{1}{2})\right)}&\leqslant C_q\cdot \left(\left\|\tilde v\right\|_{W^{1,q}\left(B(0,\frac{2}{3})\right)}+\left\|\tilde u\right\|_{W^{2,q}\left(B(0,\frac{2}{3})\right)}\right)\\
    &\leqslant C_q\cdot \left(\left\|\tilde v\right\|_{C^1(B(0,1))}+\left\|\tilde u\right\|_{C^0(B(0,1))}\right).
\end{align*}
And by bootstrapping and Sobolev embedding
\begin{align*}
    \left\|\tilde u\right\|_{C^{k+1,1-\frac{2n}{q}}\left(B(0,\frac{1}{2})\right)}&\leqslant C_{q,k}\left\|\tilde u\right\|_{W^{k+2,q}\left(B(0,\frac{1}{2})\right)}\\
    &\leqslant C_{q,k}\left(\left\|\tilde v\right\|_{W^{k,q}\left(B(0,\frac{2}{3})\right)}+\left\|\tilde u\right\|_{W^{k+1,q}\left(B(0,\frac{2}{3})\right)}\right)\\
    &\leqslant C_{q,k}\left(\left\|\tilde v\right\|_{C^k\left(B(0,1)\right)}+\left\|\tilde u\right\|_{C^0\left(B(0,1)\right)}\right).
\end{align*}
Passing to the original metric, we know that for any $k\geqslant 1$
\begin{align*}
r(x)^{k\lambda}\left\|\nabla^k u\right\|_{C^0(B_g(x,r(x)^\lambda))}\leqslant C_{k} \left(\left\| u\right\|_{C^0(B_g(x,r(x)^\lambda))}+\sum_{i=0}^{k-1}r(x)^{(i+2)\lambda}\left\|\nabla^i v\right\|_{C^0(B_g(x,r(x)^\lambda))}\right).
\end{align*}
\end{proof}


\section{Solving Poisson Equation on the Model Space}\label{Poisson equation}
In this section, following the approach of Sun-Zhang in \cite{SZ}, we will use separation of variables to solve $\Delta_{\omega_\mathcal{C}}u =v$ for $u$, $v$ functions on $\mathcal{C}$ and give some uniform estimate of our solution. We use the same notation introduced in section \ref{modelspace}. 
 	
We first notice that $\mathcal{C}$ is diffeomorphic to $Y\times \RR$ where the level set $Y =\{\xi \in \mathcal{C}\,|\,-\log |\xi|^2_{h_D} =z_0^n\}$ for a fixed $z_0>0$ and $Y$ is equipped with an $S^1$ bundle structure over $D$ and a metric $h_Y$ induced by $\omega_\mathcal{C}$. Let $0=\Lambda_0 <\Lambda_1<\cdots$ be the spectrum of the Laplacian $-\Delta_{(Y,h_Y)}$ on $Y$ with respect to the metric $h_Y$. Let $\{\psi_k\}_{k=0}^\infty$ be the corresponding eigenfunctions with $\left\|\psi_k\right\|_{L^2(Y)}=1$. They showed that $\Lambda_k = z_0^{-1}\lambda_k+nz_0^{n-1}j_k^2$ for some $\lambda_k\geqslant 0$ and $j_k\in \mathbb{N}$. Moreover, $\{\psi_k\}_{k=0}^\infty$ form an orthonormal basis of $L^2(Y)$ and each $\psi_k$ is homogeneous of degree $j_k$ under the $S^1$ action.
The product structure allows us to do Fourier expansion on $\mathcal{C}$. In particular, for any smooth function $v$ on $\mathcal{C}$, if we take $P_k(v)(z) = \int_Y v(z,y) \psi_k(y)$, we can write
\begin{align}\label{decomposition} v(z,y) = \sum_{k=0}^\infty P_k(v)(z) \cdot \psi_k(y),\end{align} 
which is convergent in $L^2$ sense. In fact, we will prove later that the convergence is in $C^k$ if $v$ has proper higher regularity estimate.
For the separated function $u(z)\psi(y)$ we have
\begin{align*}\Delta_{\omega_\Ca}u(z)\psi(y) = \frac{1}{nz^{n-1}}(u''(z)-(\lambda+\frac{j^2 n}{4} \cdot z^n) nz^{n-2} u(z))\psi(y).\end{align*}
	Moreover, \cite{SZ} proved the following:
\begin{proposition}
	Let $\left(\mathcal{C}, g_{\mathcal{C}}\right)$ be the Calabi model space, and let $u$ solve the Poisson equation $\Delta_{\mathcal{C}} u=v$ for some $v \in C^{K_0}(\mathcal{C})$ and $K_0 \in \mathbb{N}$ sufficiently large. Let $u$, $v$ have "fiber-wise" expansions as in (\ref{decomposition}). Then for every $k \in \mathbb{N}$, the coefficient functions $u_k(z)$ and $v_k(z)$ satisfy
	\begin{align}\label{ODE}
	u_k''(z)-\left(n\lambda_k+\frac{j_k^2 n^2}{4} \cdot z^n\right) z^{n-2} u_k(z)=n z^{n-1}\cdot v_k(z), \quad z \geqslant 1.
	\end{align}
	\end{proposition}

Specifically, they found a solution of $\Delta_{\omega_\Ca} u = v$ by solving ODE (\ref{ODE}). With the estimate of the solution of this equation, they showed that the $L^2$ formal solution given by $\sum u_k\psi_k$ is actually a regular solution to the Poisson equation. Similarly but directly via careful estimates, we can construct the solution of Poisson equation with respect to $\omega_\mathcal{C}$ with some weighted regularity and finer polynomial growth order:

\begin{proposition}\label{solve poisson equation}
Assume that $v$ is a function on $\Ca$ such that for any $k\in \mathbb{N}$ there exist constants $C_k$ and $\delta$ such that $|z^{\frac{k}{2}}\nabla^k v|_{\omega_{\Ca}}\leqslant C_k z^{\delta}$ on $z>C_k$. Then there exist constants $C'_k$ and a function $u:\mathcal{C}\to \RR$ such that $\Delta_{\omega_\mathcal{C}}u = v$ and 
\begin{align*}
    |u|\leqslant C_0' z^{\delta+n+1+\epsilon},\quad \left|\nabla u\right|_{\omega_\Ca} \leqslant C_1' z^{\delta+\frac{n+1}{2}+\epsilon},\quad\left|\nabla^k u\right|_{\omega_\Ca}\leqslant C'_k z^{\delta-\frac{k-2}{2}+\epsilon} 
\end{align*}
on $z>C_k'$, for any $\epsilon>0$ and any integer $k\geqslant 2$.
\end{proposition}

\begin{proof}
Let \begin{align*}
    u(z,y) = u_0(z)\psi_0 + \sum_{j =1}^\infty u_j(z)\psi_j(y)
\end{align*} be the formal solution, where $u_0$ is constructed as follows and $u_j$ is constructed in the Appendix \ref{estimate of ODE} which mainly follows from Sun-Zhang \cite{SZ}.

\textbf{Step 1:} We first show that the formal solution converges in $C^0$ sense with the polynomial order depending on the polynomial order of $v$.

For $u_0$, we have
	\begin{align*}
		u_0''(z) = nz^{n-1}P_0v(z),\quad  u_0'(z) = \int_{C_2}^z ns^{n-1}P_0v(s)ds,\quad u_0(z) = \int_{C_1}^z\left(\int_{C_2}^t ns^{n-1}P_0v(s)ds\right)dt.
	\end{align*}
We choose $C_1$ and $C_2$ here to be $1$ or $+\infty$ depending on the order of $P_0v$ to make $u_0(z)$ and $u_0'(z)$ finite with proper order. Also, the integration here is the only place that we will lose the rate $z^\epsilon$.

For $u_j$, we first prove that the projection $P_j(v)$ is well-defined and has the following estimate:
	\begin{align}\label{estimate of v_k}
	|v_j(z)| = \left|P_j(v)(z)\right|&=\left\vert\int_{Y}\frac{\Delta^{K_0}_{h_Y}v(z,y)\psi_j}{(\Lambda_j)^{K_0}} \mathrm{dVol_Y}\right\vert\leqslant \frac{\left\|v(z,\cdot)\right\|_{C^{2K_0}(Y,h_Y)}}{(\Lambda_j)^{K_0}}\leqslant C_{K_0} \frac{z^\delta}{\Lambda_j^{K_0}}
	\end{align} for any $K_0$ and $z>C$. Consequently, the solution $u_j$ constructed in (\ref{estimate of u wrt v 0}) and (\ref{estimate of u wrt v 1}) for the equation (\ref{ODE}) has the following $C^0$ bound for $z>C_{n,\delta,\lambda_1}$:
\begin{align}\label{estimate of u_k}  
	\left|u_j(z)\right|\leqslant \frac{1}{(\Lambda_j)^{K_0}}\frac{C_{n,\delta,\lambda_1,K_0}z^{\delta+1}}{n\lambda_1}.
		\end{align}
Here the constant $C$ is uniform for $j$ and $z$ and only depends on $K_0$ and $n$.We also have the uniform estimate of $\psi_j$ as the eigenfunctions of $-\Delta_{\omega_Y}$ on $Y$ by its eigenvalues showed in Sun-Zhang \cite[Lemma 5.1.]{SZ}:
	\begin{align}\label{estimate of psi_k}
	\left\|\psi_j\right\|_{C^k\left(Y\right)} &\leqslant C_{k,Y} \cdot\left(\Lambda_j\right)^{\frac{n+k}{2}}.
    \end{align} 
Combine (\ref{estimate of u_k}) and (\ref{estimate of psi_k}), we get
\begin{align*}
	|\sum_{j =1}^\infty u_j\psi_j| &\leqslant \sum_{j=1}^{\infty}\left|u_j(z)\right| \cdot\left|\psi_j(y)\right| \leqslant C_{n,\delta,\lambda_1,K_0,Y} \sum_{j=1}^{\infty} \frac{z^{\delta+1}}{\left(\Lambda_j\right)^{K_0-\frac{n}{2}}}.
\end{align*}
		
Weyl's law gives the bound of $\Lambda_j$ with $C_Y^{-1} j^{\frac{2}{2n-1}} \leq\left|\Lambda_j\right| \leqslant C_Y j^{\frac{2}{2n-1}}$ for some $C_Y$ only depend on $(Y,h_Y)$. Take $K_0 =2n $ we can conclude that the summation converges.\\

\textbf{Step 2:} We prove that $u$ is smooth. We mostly follow the proof in Sun-Zhang \cite[ Proposition 6.2.]{SZ}. Let 
\begin{align*}
U_N = \sum_{j=0}^N u_j \psi_j,\quad V_N = \sum_{j=0}^N v_j \psi_j. 
\end{align*}
Then we have $\Delta_{\omega_{\mathcal{C}}} U_N = V_N$.\\
We first show that $V_N$ has $C^k$ bound independent of $N$. In fact, we have the higher regularity estimate for $v_j$ as in (\ref{estimate of v_k}):
\begin{align*}
    | \nabla^k v_j(z)| &= \left|\nabla^k P_j(v)(z)\right|=\left\vert\nabla^k\int_{Y}\frac{\Delta^{K_0}_{h_Y} v(z,y)\psi_j}{(\Lambda_j)^{K_0}} \mathrm{dVol_Y}\right\vert\leqslant C_{z,Y,K_0,k}\frac{1}{\Lambda_j^{K_0}}.
\end{align*}
Then given by the estimate of $\psi_j$ (\ref{estimate of psi_k}), for any integer $j$, we know that
\begin{align*}
    \sum_{j=1}^N\left|\nabla^k (v_j\psi_j) \right|\leqslant  C_{z,Y,K_0,k} \sum_{j=1}^N\Lambda_j^{\frac{n+k}{2}-K_0}.
\end{align*}
So again by Weyl's law $V_N$ converges to $v$ as a $C^k$ function. Now we can prove that $U_N$ also have the uniform $C^k$ bound with respect to $N$ via local elliptic estimates.

For any fixed point $x\in \mathcal{C}$, we consider the ball $B_{\omega_\mathcal{C}}(x,1)$. Then for any $p>0$, there exists a constant $C_{p,x}$ such that
\begin{align*}
    \left\|U_N\right\|_{W^{2,p}(B_{\omega_\mathcal{C}}(x,\frac{1}{2}))}\leqslant C_{p,x}\cdot (\left\|V_N\right\|_{C^{0}(B_{\omega_\mathcal{C}}(x,1))}+\left\|U_N\right\|_{C^0(B_{\omega_\mathcal{C}}(x,1))})\leqslant C_{p,x}.
\end{align*}
By bootstrapping and Sobolev embedding, for any $k\geqslant 0$ and $p>0$, there exists a constant $C_{p,x}$ such that
\begin{align*}
    \left\|U_N\right\|_{C^{k+1,1-\frac{2n}{p}}(B_{\omega_\mathcal{C}}(x,\frac{1}{2}))}&\leqslant C_{p,x}\left\|U_N\right\|_{W^{k+2,p}(B_{\omega_\mathcal{C}}(x,\frac{1}{2}))}\\
    &\leqslant C_{p,x} \left(\left\|V_N\right\|_{W^{k,p}(B_{\omega_\mathcal{C}}(x,1))}+\left\|U_N\right\|_{W^{k,p}(B_{\omega_\mathcal{C}}(x,1))}\right)\leqslant C_{p,k,x}.
\end{align*}
Consequently we have $U_N$ converges to $u$ in $C^{k}$. Since $x$ is arbitrary, we know that $u$ is smooth on $\mathcal{C}$.\\

\textbf{Step 3.} We can now give global bound on the $C^1$, $C^2$ and higher regularity of $u$. We treat $u_0$ and $u- u_0$ seperately.

For $u_0$, we have explicit estimate by computation of the Christoffel symbol under the following holomorphic coordinate: Let $\pi:\mathcal{C}\rightarrow D$ be the projection map. For any point $\xi \in \Ca$ we take the local holomorphic coordinate $\underline{z} =(z_1, z_2,\ldots ,z_{n-1})$ on $B\subset D$ around the point $\pi(\xi)$. Take $\xi_0$ be a local holomorphic section of $N_D$ such that $\left|\xi_0\right|_{h_D} = e^{-\varphi}$, where $i\pp\bp \varphi = \omega_D$ with $\varphi(0)=0$, $\nabla\varphi(0) = 0$. Then $\xi = \xi_0\cdot w$ where $w$ is the fiber coordinate. Recall that $z=\left(-\log |\xi|_{h_D}^2\right)^{\frac{1}{n}}$, $w = e^{i\theta-t+\varphi/2}$ and 
$$\omega_{\mathcal{C}}=z \pi^*\omega_D+\frac{1}{nz^{n-1}}\cdot i\cdot\left(\frac{d w}{w}-\partial \varphi\right) \wedge\left(\frac{d \bar{w}}{\bar{w}}-\bar{\partial} \varphi\right).$$
Under this coordinate, we can prove by induction that 
\begin{align}
    &\left|\nabla^k dz_i\right|_{\omega_\Ca}\leqslant C z^{-\frac{k+1}{2}}, \text{ for any } k\geqslant  0, \nonumber\\
    &\left|\frac{dw}{w}\right|_{\omega_\Ca}\leqslant z^{\frac{n-1}{2}},\quad\left|\nabla^k \frac{dw}{w}\right|_{\omega_\Ca}\leqslant C z^{-\frac{k+1}{2}},\text{ for any } k\geqslant 1,\nonumber\\ 
    &\left|dt\right|_{\omega_\Ca}\leqslant z^{\frac{n-1}{2}}, \quad\left|\nabla^k t\right|_{\omega_\Ca}\leqslant z^{-\frac{k}{2}},\quad\left|\nabla^k t^a\right|_{\omega_\Ca}\leqslant z^{n(a-1)-\frac{k}{2}}, \text{ for any } k\geqslant 2. \label{estimate of t}
\end{align}
Now we can estimate the higher derivative of $u_0$:
\begin{align*}
    \nabla^ku_0 &= \sum_{j = 1}^k u_0^{(j)}(z)\sum_{\substack{i_1+\cdots+i_j = k\\i_1>0,\cdots,i_j>0}}C_{i_1,\cdots,i_j} \nabla^{i_1}z\otimes\nabla^{i_2}z\otimes\cdots\otimes\nabla^{i_j}z,\\ 
    &=u_0'(z)\nabla^kz+\sum_{j = 2}^k u_0^{(j)}(z)\sum_{\substack{i_1+\cdots+i_j = k\\i_1>0,\cdots,i_j>0}} C_{i_1,\cdots,i_j}\nabla^{i_1}z\otimes\nabla^{i_2}z\otimes\cdots\otimes\nabla^{i_j}z,\\
    &=u_0'(z)\nabla^kz+\sum_{j = 2}^k \left(nz^{n-1}P_0v(z)\right)^{(j-2)}\sum_{\substack{i_1+\cdots+i_j = k\\i_1>0,\cdots,i_j>0}} C_{i_1,\cdots,i_j}\nabla^{i_1}z\otimes\nabla^{i_2}z\otimes\cdots\otimes\nabla^{i_j}z,
\end{align*}
where $f^{(k)}$ refers to the higher derivative of the function $f$. Since
\begin{align*}
    \left(nz^{n-1}P_0v(z)\right)^{(j-2)} = \sum_{l=0}^{j-2}C_{l,j,n}z^{n-j+l+1}P_0v^{(l)} \text{ and }
    \left|P_0v^{(l)}(z)\right| \leqslant C_Y\left|\nabla^lv\right|_{\omega_\Ca}\cdot\left|dz\right|^{-l}_{\omega_\Ca}\leqslant C_Yz^{\delta+\frac{(n-2)l}{2}},
\end{align*}
we know that 
\begin{align*}
    \left|\nabla^ku_0\right|_{\omega_\Ca} &\leqslant C_{n,k,Y}z^{\delta+1-\frac{k}{2}+\epsilon}+ C_{n,k,Y}\sum_{j = 2}^k \sum_{l=0}^{j-2}z^{n-j+1+\delta+\frac{nl}{2}}\sum_{\substack{i_1+\cdots+i_j = k\\i_1>0,\cdots,i_j>0}}\left|\nabla^{i_1}z\right|_{\omega_\Ca}\otimes\left|\nabla^{i_2}z\right|_{\omega_\Ca}\otimes\cdots\otimes\left|\nabla^{i_j}z\right|_{\omega_\Ca}\\
    &\leqslant C_{n,k,Y}z^{\delta+1-\frac{k}{2}+\epsilon}+ C_{n,k,Y}z^{1+\delta}\sum_{j = 2}^k \sum_{\substack{i_1+\cdots+i_j = k\\i_1>0,\cdots,i_j>0}}\left|z^{\frac{n-2}{2}}\nabla^{i_1}z\right|_{\omega_\Ca}\otimes\left|z^{\frac{n-2}{2}}\nabla^{i_2}z\right|_{\omega_\Ca}\otimes\cdots\otimes\left|z^{\frac{n-2}{2}}\nabla^{i_j}z\right|_{\omega_\Ca}.
\end{align*}
By $\left|\nabla z\right|_{\omega_\Ca}\leqslant z^{-\frac{n-1}{2}}$ and $\left|\nabla^i z\right|_{\omega_\Ca}\leqslant z^{-\frac{i}{2}-n+1}$ for $i\geqslant 2$, we have 
\begin{align*}
    \sum_{\substack{i_1+\cdots+i_j = k\\i_1>0,\cdots,i_j>0}}\left|z^{\frac{n-2}{2}}\nabla^{i_1}z\right|\otimes\left|z^{\frac{n-2}{2}}\nabla^{i_2}z\right|\otimes\cdots\otimes\left|z^{\frac{n-2}{2}}\nabla^{i_j}z\right|\leqslant C_kz^{-\frac{k}{2}},
\end{align*}
Consequently, we have 
\begin{align*}
    \left|\nabla u_0\right|_{\omega_\Ca} \leqslant C_{n,Y}z^{\delta+\frac{n+1}{2}+\epsilon}.
\end{align*}
and for $k\geqslant 2$ 
\begin{align*}
    \left|\nabla^ku_0\right|_{\omega_\Ca} \leqslant C_{n,k,Y}z^{\delta+1-\frac{k}{2}+\epsilon}.
\end{align*} 

For $u-u_0$ we use our $C_0$ bound of $u-u_0$ and do elliptic estimate around a point $x \in \{z= z_0\}$ by Proposition \ref{uniform schauder estimates}. We have the uniform estimate of $u-u_0$ on $\{z\geqslant C'\}$:
\begin{align*}
z^{\frac{k}{2}}\left\|\nabla^k (u-u_0)\right\|_{C^0(B_g(x,\sqrt{z}))}\leqslant C_{k} \left(\left\| u-u_0\right\|_{C^0(B_g(x,\sqrt{z}))}+\sum_{i=0}^{k-1}z^{\frac{i+2}{2}}\left\|\nabla^i v-P_0(v)\right\|_{C^0(B_g(x,\sqrt{z}))}\right),
\end{align*}
which yields\begin{align*}
    \left|z^{\frac{k}{2}}\nabla^k (u-u_0)\right|_{\omega_\Ca}\leqslant C_k' z^{\delta+1}.
\end{align*}
Together with the estimate of $u_0$, we have
\begin{align*}
    |u|\leqslant C_0' z^{\delta+n+1+\epsilon},\quad \left|\nabla u\right|_{\omega_\Ca} \leqslant C_1' z^{\delta+\frac{n+1}{2}+\epsilon},\quad\left|\nabla^k u\right|_{\omega_\Ca}\leqslant C_k' z^{\delta-\frac{k-2}{2}+\epsilon} \text{ for } k\geqslant 2.
\end{align*}
\end{proof}

\begin{remark}
	We see from the proof that the main term in $C^0$ and $C^1$ estimate of the solution $u$ is the fiber direction $u_0$. However, for the higher estimate $C^k$ where $k\geqslant 3$, they will give the same order contribution.
\end{remark}

\section{Improve the Decay of the Ricci potential}\label{decay faster than quadratic}

Now we look back at the quasi-projective manifold $X$ and the class $\mathfrak{k}\in H^2(X)$. In this section we are going to find a good representative form $\beta$ inside the class $\mathfrak{k}\in H^2(X)$. We look at its behavior on the model space and then find some function $U$ by finite step iteration such that $$\frac{(\Phi^*\beta+i\pp\bp U)^n}{\omega_\Ca^n}-1$$ has faster decay rate.

\subsection{A good representative}
\begin{lemma}\label{a good representative}
	For any $\mathfrak{k} \in H^2(X)$, there exists a closed $(1,1)$-form $\beta$ on $M$ such that $[\beta|_{X}] = \mathfrak{k}$.
\end{lemma}
\begin{proof}
By discussion in Section 2 of \cite{hsvz2} we know $M$ is weak Fano, and consequently simply connected by \cite{TakayamaS2000Scow}. Consider the exact sequence $0\to \mathcal{O}_M(-D)\to \mathcal{O}_M\to \mathcal{O}_D\to 0$, we have long exact sequence 
$$\cdots\to H^1(M,\mathcal{O}_M)\to H^1(D,\mathcal{O}_D)\to H^2(M,\mathcal{O}_M(-D))\to\cdots$$
On the other hand, by Serre duality we have$$H^2(M,\mathcal{O}_M(-D))\simeq H^2(M,K_M)\simeq H^{n,n-2}(M,-K_M).$$ Apply Kawamata–Viehweg vanishing theorem to the nef and big line bundle $-K_M$ so we know that $H^{n,p}(M,-K_M)=H^{0,p}(M)= 0$ for $p\geqslant 1$. So when $n\geqslant 3$, we have $H^1(D,\mathbb{C}) =0$. The long exact sequence given by the excision theorem and Thom-Gysin sequence
$$H^0(D)\to H^2(M)\to H^2(X)\to H^1(D)\to \cdots $$
 yields that the restriction map $j^*:H^2(M)\to H^2(X)$ induced by $j:X\to M$ is surjective with $
\mathrm{dim}\,\mathrm{Ker} j^* = 1$, generated by $c_1(-K_M)$. Hence there exists a closed $(1,1)$-form $\beta$ on $M$ such that $[\beta|_{X}] = \mathfrak{k}$.
\end{proof}
\begin{remark}
	Here we use the fact that $\mathrm{dim}_\mathbb{C}X = n\geqslant 3$ to apply Kawamata–Viehweg vanishing theorem. When $n=2$ we have $H^2(M,K_M)\simeq H^{0,0}(M)\simeq \mathbb{C}$, so $H^1(D,\mathbb{C})$ may not vanish.
\end{remark}
The global $(1,1)$ form $\beta$ on $M$ satisfies the following property on the end:
\begin{proposition}
    Let $\Phi: \mathcal{C}\setminus \mathcal{K}\to X\setminus K$ be the fixed diffeomorphism and let $p:\Ca\to D$ be the projection map. Then $\left|\Phi^*\beta-p^*(\left.\beta\right|_D)\right|_{\omega_\mathcal{C}}=O(e^{-(\frac{1}{2}-\epsilon) z^n})$.
\end{proposition}

\begin{proof}
    Let $p$ be a fixed point in $D$. Let $(w,\underline{z}) = (w, z_1, \cdots, z_n)$ be local holomorphic coordinates around this point such that $D$ is given by $\{w=0\}$. Then $(w,\underline{z})$ can also be seen as a group of local holomorphic coordinates around $p$ in $\Ca$ where $w$ represent the fiber direction. 
    We can express $\beta$ locally around $p$ on $M$ as 
    \begin{align*}
    \beta = &\sum_{i,j=1}^{n-1}f_{i\olsi j}dz_i\wedge d\bar z_j +\sum_{i=1}^{n-1}f_{i\bar 0} dz_i\wedge d\bar w+ \sum_{i=1}^{n-1}f_{0\olsi i}  d w\wedge d\bar z_i +f_{00}dw\wedge d\bar w,\\
        \Phi^*\beta = &\sum_{i,j=1}^{n-1}f_{i\olsi j}(\Phi) dz_i(\Phi)\wedge \Phi^* (J_X) d z_j(\Phi) +\sum_{i=1}^{n-1}f_{i\bar 0}(\Phi) dz_i(\Phi)\wedge  \Phi^* (J_X) d w(\Phi)\\
        &+ \sum_{i=1}^{n-1} f_{0\olsi i}(\Phi) d w(\Phi)\wedge \Phi^*(J_X) d z_i(\Phi)+f_{0\bar 0}(\Phi) dw(\Phi)\wedge  \Phi^* (J_X) d w(\Phi).
    \end{align*}

    Notice that we have the estimate of $\left|\nabla^k_{g_\Ca}w\right|=O(e^{-(\frac{1}{2}-\epsilon)z^n})$, $\left|\nabla^k_{g_\Ca}z_i\right|=O(1)$, $\left.\Phi\right|_D = \mathrm{Id}$, and the complex structure $\Phi ^*J_X-J_C$ is exponentially decay as in Proposition \ref{exponential close}, we know that on $\Ca$ we have
    \begin{align*}
        \left.(\Phi^*\beta)\right|_D = &\sum_{i,j=1}^{n-1}f_{i\olsi j}(0,\underline{z})dz_i\wedge J_\Ca d z_j +O(e^{-(\frac{1}{2}-\epsilon)z^n}).
    \end{align*}
    On the other hand, we know that $\left.p^*(\left.\beta\right|_D)\right|_D = \sum_{i,j=1}^{n-1}f_{i\olsi j}(0,z_i)dz_i\wedge d\bar z_j$. So $\Phi^*\beta-p^*(\left.\beta\right|_D)$ extends to a smooth form on $N_D$ vanishing on the zero section $D$, which yields $\left|\Phi^*\beta-p^*(\left.\beta\right|_D)\right|_{\omega_\mathcal{C}}=O(e^{-(\frac{1}{2}-\epsilon) z^n})$.
\end{proof}

\subsection{Iteration process} With this exponential closeness, we can view $\Phi^*\beta$ as a $(1,1)$-form on $\Ca$ with only horizontal direction component. This will greatly simplify our computation below.
\begin{definition}
	Let $\eta$ be a $(1,1)$-form on $\mathcal{C}$. We define $F(\eta):=1-\frac{(\omega_\mathcal{C}+\eta)^n}{\omega_\mathcal{C}^n}$, called \emph{$\omega_\mathcal{C}$-potential}.
\end{definition}
\begin{definition}\label{iteration process} With the same $\beta$ as before, we define by iteration
\begin{align*}
	F_0:= F(p^*(\left.\beta\right|_D)),\quad F_j:= F(p^*(\left.\beta\right|_D)+i\pp\bp U_j) = 1-\frac{(\omega_\mathcal{C}+p^*(\left.\beta\right|_D)+i\pp\bp U_j)^n}{\omega_\mathcal{C}^n}
\end{align*}
where \begin{align*}
U_0 = 0,\quad U_{j} = U_{j-1}+u_{j}, \quad \Delta_{\omega_\mathcal{C}}u_{j} =F_{j-1}.
\end{align*}
Here $u_{j}$ is the solution constructed in Proposition \ref{solve poisson equation}. We will prove in Proposition \ref{improve the decay of f_0} that the derivative of $F_j$'s satisfy the decay condition required for $v$ in Proposition \ref{solve poisson equation} so this iteration process works.
\end{definition}

\begin{proposition}\label{improve the decay of f_0}
	With $U_j$'s and $u_j$'s defined in \ref{iteration process}, we have $F_n= F(p^*(\left.\beta\right|_D)+i\pp\bp U_{n})$ decays faster than $z^{-n}$. More precisely, for any positive integer $j$ and $k$
	\begin{equation}\label{higher regularity of F_j}
   \|z^\frac{k}{2}\nabla^kF_j\|_{\omega_\Ca}\leqslant C_{k,j} z^{-j-1+\epsilon}.
	\end{equation}
\end{proposition}

\begin{proof}
We prove (\ref{higher regularity of F_j}) by induction.

For $F_0$, we can see this estimate directly follows from computation:
\begin{align*}
    F_0= 1-\frac{(\omega_\mathcal{C}+p^*(\left.\beta\right|_D))^n}{\omega_\mathcal{C}^n}
    =\sum_{j=1}^n\frac{n-j}{nz^j}\cdot p^*(\frac{\left.\beta\right|_D^j\wedge \omega_D^{n-j-1}}{\omega_D^{n-1}}).
\end{align*}
By (\ref{estimate of t}) we know that $|z^\frac{k}{2}\nabla^kF_0|_{\omega_\Ca}\leqslant C(k) z^{-1}$, for any positive integer $k$. So when $j=0$ (\ref{higher regularity of F_j}) holds.

Assume (\ref{higher regularity of F_j}) holds for $i\leqslant j$, i.e. $|z^\frac{k}{2}\nabla^kF_i|_{\omega_\Ca}\leqslant C(k,i) z^{-i-1}$, for any $k\geqslant 0$. 
By straightforward computation,  
\begin{align*}
F_{j+1}&=- \frac{ni\pp\bp u_{j+1}\wedge\sum_{k=1}^{n-1}{n-1 \choose k}(p^*(\left.\beta\right|_D)+i\pp\bp U_j)^k\wedge\omega_{\mathcal{C}}^{n-k-1}}{\omega_{\mathcal{C}}^n}\\
 &\qquad\qquad\qquad\qquad-\frac{\sum_{k=2}^{n}{n \choose k}\left(i\pp\bp u_{j+1}\right)^k\wedge\left(\omega_\mathcal{C}+p^*(\left.\beta\right|_D)+i\pp\bp U_j\right)^{n-k}}{\omega_{\mathcal{C}}^n}.
\end{align*}
Actually, the function in each term is of the from 
\begin{align}\label{each term in F_j+1}
	\frac{Ci\pp\bp u_{j+1}\wedge \bigwedge_{q\in Q} i\pp\bp u_q\wedge p^*(\left.\beta\right|_D)^m\wedge\omega_\mathcal{C}^{n-m-|Q|-1}}{\omega_\mathcal{C}^n},
\end{align} 
where $Q$ is a set with repeated elements from $\{1,2,\cdots, j+1\}$, $m$ is a non-negative integer and positive when $j+1\notin Q$. By Proposition \ref{solve poisson equation} we know that 
$\left|z^{\frac{k}{2}}\nabla^k u_{j+1}\right|_{\omega_{\mathcal{C}}}\leqslant C_{k,j}z^{-j+\epsilon}$.
It is easier to deduce the bound for (\ref{each term in F_j+1}) by passing to the rescaled metric $\tilde\omega_{\Ca}$ on $B_{\tilde\omega_\Ca}(x,1)$, where $z(x)= z_0$, $\tilde\omega_{\Ca} = \frac{\omega_{\Ca}}{z_0}$. We have the uniform weighted bound for each term in the wedge product
\begin{align*}
    \left\|\nabla^k i\pp\bp u_q\right\|_{C^0(B_{\tilde\omega_\Ca}(x,1))}\leqslant C_{k,q} z_0^{-q+1+\epsilon},
    \left\|\nabla^k i\pp\bp u_{j+1}\right\|_{C^0(B_{\tilde\omega_\Ca}(x,1))}\leqslant C_{k,j} z_0^{-j+\epsilon},\\
    \left\|\nabla^k p^*(\left.\beta\right|_D)\right\|_{C^0(B_{\tilde\omega_\Ca}(x,1))}\leqslant C_k.
    \end{align*}

If we consider the scaled metric $\tilde\omega_\Ca$ our function (\ref{each term in F_j+1}) becomes
\begin{align*}
    \frac{Ci\pp\bp u_{j+1}\wedge \bigwedge_{q\in Q} i\pp\bp u_q\wedge p^*(\left.\beta\right|_D)^m\wedge\tilde\omega_\mathcal{C}^{n-m-|Q|-1}}{z_0^{m+|Q|+1}\tilde\omega_\mathcal{C}^n}.
\end{align*}
So we have
\begin{align*}
    \left\|\nabla^k\frac{Ci\pp\bp u_{j+1}\wedge \bigwedge_{q\in Q} i\pp\bp u_q\wedge p^*(\left.\beta\right|_D)^m\wedge\tilde\omega_\mathcal{C}^{n-m-|Q|-1}}{z_0^{m+|Q|+1}\tilde\omega_\mathcal{C}^n}\right\|_{C^0(B_{\tilde\omega_\Ca}(x,1))}\leqslant C_{k,j,Q,m}z_0^{-j-2+\epsilon}.
\end{align*}
Consequently, 
\begin{align*}
    \left\|z_0^{\frac{k}{2}}\nabla^k F_{j+1}\right\|_{C^0(B_{\omega_\Ca}(x,\sqrt{z_0}))} = \left\|\nabla^k F_{j+1}\right\|_{C^0(B_{\tilde\omega_\Ca}(x,1))}\leqslant C_{k,j}z_0^{-j-2+\epsilon}.
\end{align*}
So we finish the proof of (\ref{higher regularity of F_j}). Specially,
\begin{align*}
    \left|z^{\frac{k}{2}}\nabla^k F_{n}\right|_{\omega_\Ca} \leqslant C_kz^{-n-1+\epsilon}.
\end{align*}
\end{proof}

\section{The Integral Condition}\label{integral condition}

For the convenience of statement, let us introduce the following notation.
\begin{definition}
	For an $(n,0)$ form $\Omega$, we say that a (1,1)-form $\alpha$ is \emph{$\Omega$-compatible} if $\int_X \Omega\wedge\ols\Omega -\alpha^n=0$.
\end{definition}
\begin{remark}
    Since $\Omega\wedge\ols\Omega$ and $\alpha^n$ are not integrable for most of the time, this integration identity means that the function $f = \frac{\alpha^n}{\Omega\wedge\ols\Omega}-1$ satisfies $\int_X f \Omega\wedge\ols\Omega = 0$.
\end{remark}

In this section, we will show that by adding a suitable potential we can make $\beta+i\pp\bp U$ to be $\Omega_X$-compatible.
	
\begin{proposition}\label{compatible condition}
There exists a smooth function $\tilde{U}$ on $X$ such that $\beta+i\pp\bp \tilde{U}$ is an $\Omega_X$-compatible K\"ahler form. Meanwhile, we have that \begin{align*}
    \left|\Phi^*\tilde{U} - (z^{n+1} +U_n+\lambda z)\right| \leqslant Ce^{-\delta z^n}
\end{align*} when $z>C$ for some constant $C>0$ and $\lambda\in\RR$.
\end{proposition}

\begin{proof}
We first show that we can find $U$ such that $(\beta+i\pp\bp U+\omega_X)^n$ is integrable on the end. Recall that we have $U_j = \sum_{p=1}^j u_p$ and $u_j$ are functions on $\mathcal{C}$ such that 
$$\Delta_{\omega_\mathcal{C}} u_{j+1} = 1-\frac{\left(p^*(\left.\beta\right|_D)+\omega_\Ca+i\pp\bp U_j\right)^n}{\omega_\Ca^n} = F_j.$$ 
We know that the following integration is finite since $|F_n|\leqslant C_nz^{-n-1+\epsilon}$:
\begin{align*}
	&\quad \left|\int_{\mathcal{C}\setminus\mathcal{K}}(\Phi^*\beta+\omega_\mathcal{C}+i\pp\bp U_{n})^n-\omega_{\mathcal{C}}^n\right| \\
	&\leqslant \int_{\mathcal{C}\setminus\mathcal{K}}|F_{n}|(i\pp\bp t)^{n-1}dt\wedge d^c t +\left|\int_{\mathcal{C}\setminus\mathcal{K}}(\Phi^*\beta+\omega_\mathcal{C}+i\pp\bp U_{n})^n-(p^*(\left.\beta\right|_D)+\omega_{\mathcal{C}}+i\pp\bp U_{n})^n\right|\\
	&=\int_{\{t=T\}}nT^{\frac{-1+\epsilon}{n}}(i\pp\bp t)^{n-1}\wedge d^c t+C
    <+\infty.
 \end{align*}

Since we have the exponentially closed estimate between $X\setminus K$ and $\mathcal{C}\setminus\mathcal{K}$, we will have $$\int_{X\setminus K}\left(\beta +i\pp\bp (\Phi^{-1})^*U_{n}+\omega_X \right)^n-\omega_{X}^n = \int_{\mathcal{C}\setminus\mathcal{K}}(\Phi^*\beta +i\pp\bp U_n+\omega_\Ca)^n-\omega_{\mathcal{C}}^n +O(e^{-\delta z^n}) <+\infty.$$

Now we can construct the K\"ahler potential following the construction in Hein-Sun-Viaclovsky-Zhang \cite[Lemma 2.7.]{hsvz2}:

The ampleness of $N_D$ implies that $X$ is 1-convex. Hence by Remmert reduction we know that $-K_{M}$ is semi-ample, we denote its non-ample locus by $E$. Recall that $[\beta ]^p\cdot Y>0$ for any $p$-dimensional compact subvariety $Y$ in $X$, by the generalized Demailly-P\u{a}un criterion in \cite{collins2016singular} we know that there exists a smooth function $u_0$ on $X$ such that $\beta+i\pp\bp u_0$ is positive on the neighborhood $U$ of $E$. Let $\chi_0$ be a smooth function on $X$ support on $U$ and $\chi_0 = 1$ on $E$. 
Then $\beta+ i\pp\bp\left(\chi_0\cdot u_0\right)$ is positive around $E$ and  $i\pp\bp\left(\chi_0\cdot u_0\right)$ is supported on $U$.

Let $\mathfrak{t}=-\log|S|^2_{h_M}$, then the curvature form $i\pp\bp \mathfrak{t}\geqslant 0$ and $i\pp\bp \mathfrak{t}>0$ on $X\setminus E$. Let $\chi_1$ be a smooth cutoff function on $[0,+\infty)$ such that $\chi_1 = 1$ on $[0,1]$ and $\chi_1 = 0$ on $[2,+\infty)$. Then $i\pp\bp (\chi_1(\tfrac{\mathfrak{t}}{C_2})\cdot \mathfrak{t}) $ is positive on $\{\mathfrak{t}\leqslant C_2\}$ and supported on $\{\mathfrak{t}\leqslant 2 C_2\}$.

Let $\rho_A$ be a smooth convex function on $\RR$ with $\rho_A(x) = \tfrac{2A}{3}$ on $(-\infty, \tfrac{A}{2}]$ and $\rho_A(x) = x$ on $(A, +\infty)$. Then we can obtain that $i\pp\bp(\rho_{C_3}(\mathfrak{t}^{\tfrac{n+1}{n}}))\geqslant 0$ on $X$ and $i\pp\bp(\rho_{C_3}(\tfrac{n}{n+1}\mathfrak{t}^{\tfrac{n+1}{n}})) =i\pp\bp(\tfrac{n}{n+1}\mathfrak{t}^{\tfrac{n+1}{n}}) $ on $\{\mathfrak{t}\geqslant 2 C_3^{\tfrac{n}{n+1}}\}$.

Let \begin{align*}
	\beta_1=\beta +i\pp\bp\left(\chi_0\cdot u_0+C_1\chi_1\left(\frac{\mathfrak{t}}{C_2}\right)\cdot \mathfrak{t} +\rho_{C_3}\left(\tfrac{n}{n+1}\mathfrak{t}^{\frac{n+1}{n}}\right)\right).
\end{align*} 
By our choice of $u_0$, $\beta+i\pp\bp u_0$ is positive around $E$. By choosing $C_1$ and $C_2$ large we can make $\beta_1$ is positive on $U$. Then choosing $C_2$ large enough depending on $C_1$ and $C_3$ we have that $\beta_1$ is positive on $\{C_2\leqslant \mathfrak{t}\leqslant 2 C_2\}$ hence K\"ahler on $X$.

Then we can glue our perturbation function $U_n$ via a cut-off function $\chi_2$ supported outside a compact set $K'$ with $\chi_2 = 1$ outside a open neighborhood $U$ of $K'$ and let 
$$\beta_2(\lambda) = \beta_1+i\pp\bp\left(\chi_2\cdot\left( (\Phi^{-1})^*(U_{n}+\lambda z)\right)\right).$$
Our goal next is to find suitable $\lambda$ and $K'$ such that $\beta_2(\lambda)$ is K\"ahler and $\Omega_X$-compatible.

 Let us first show that the $\Omega_X$-compatible condition is a linear equation of $\lambda$ and only the constant term depends on the choice of $\chi_2$. By our previous estimate of $U_{n}$ we know as a starting point that $\int_X\beta_2(0)^n -\Omega_X\wedge\ols\Omega_X =C$ is finite. The $\Omega_X$-compatible condition becomes
\begin{align*}
	0=&\int_X \beta_2(\lambda)^n-\Omega_X\wedge\ols\Omega_X = \int_X \beta_2(\lambda)^n-\beta_2(0)^n +\int_X\beta_2(0)^n -\Omega_X\wedge\ols\Omega_X\\
	&=\lim_{\varepsilon\to 0}\int_{\{(\Phi^{-1})^*t\leqslant -\log \varepsilon\}} \lambda i\pp\bp\left(\chi_2\cdot(\Phi^{-1})^*z\right)\wedge \sum_{k=0}^{n-1}{n\choose k}\beta_2(0)^k\wedge i\pp\bp\left(\lambda\chi_2\cdot(\Phi^{-1})^*z\right)^{n-k-1} +C\\
	& =\lim_{\varepsilon\to 0}\int_{\{t = -\log \varepsilon\}}\lambda d^ct \wedge \frac{1}{nz^{n-1}}\sum_{k=0}^{n-1}{n\choose k}(\omega_\mathcal{C}+\beta+i\pp\bp U_{n})^k\wedge (\lambda i\pp\bp z)^{n-k-1} +C.
\end{align*}
Expanding the terms in the bracket, we notice that only $\omega_\mathcal{C}^{n-1}$ remains non-vanishing after we take the limit, so we have the equation
\begin{align*}
	0&=\lim_{\varepsilon\to 0}\int_{\{t = -\log \varepsilon\}}\lambda d^ct \wedge (i\pp\bp t)^{n-1}+C = \lambda \cdot\int_D \omega_D^{n-1}+C.
\end{align*}
This is a linear equation on $\lambda$. On the other hand, we also notice by the previous computation that $\chi_2$ does not affect the integral, so we can choose $\lambda_0$ first to satisfy the integral condition and then choose $K'$ large enough such that $\beta_2(\lambda_0)$ is K\"ahler.
So by choosing $$\tilde{U} =\chi_0\cdot u_0+C_1\chi_1\left(\tfrac{\mathfrak{t}}{C_2}\right)\cdot \mathfrak{t} +\rho_{C_3}\left(\tfrac{n}{n+1}\mathfrak{t}^{\frac{n+1}{n}}\right)+ \chi_2\cdot\left(\left(\Phi^{-1}\right)^*(U_{n}+\lambda z)\right)$$ we finish our proof.
\end{proof}

In order to apply Tian-Yau-Hein's package, we need to repeat the iteration process for one more step such that the $\omega_\Ca$-potential of $\beta+i\pp\bp U$ decays faster than $r^{-2}$.

\begin{proposition}\label{final potential}
	Furthermore, we can construct a K\"ahler $\Omega_X$-compatible form $\beta+i\pp\bp U$ on $X$ such that 
    \begin{align*}
    \left|1-\frac{(\beta+i\pp\bp U)^n}{\Omega_X\wedge\ols\Omega_X}\right|\leqslant C r^{-2-\epsilon},\qquad \left|\,\omega_\mathcal{C}-\Phi^*(\beta+i\pp\bp U)\right|_{\omega_\Ca}\leqslant C z^{-1},
    \end{align*} where $r$ is the distance function to some point $p \in X$ under metric $\beta+i\pp\bp U$,  $C>0$.
\end{proposition}
 
\begin{proof}
    Let $u_{n+1} = \lambda z$, $U_{n+1} = U_{n}+u_{n+1}$. Let $u_{n+2}$ be the solution of $\Delta_{\omega_\mathcal{C}}u_{n+2} = F_{n+1}$ constructed in Proposition \ref{solve poisson equation}.
	\begin{align*}
		F_{n+1} &=F_n -\frac{ni\pp\bp \lambda z\wedge\sum_{k=1}^{n-1}{n-1 \choose k}(p^*(\left.\beta\right|_D)+i\pp\bp U_{n})^k\wedge\omega_{\mathcal{C}}^{n-k-1}}{\omega_{\mathcal{C}}^n}\\
        &\qquad\qquad\qquad\qquad-\frac{\sum_{k=2}^{n}{n \choose k}\left(i\pp\bp\lambda z\right)^k\wedge\left(\omega_\mathcal{C}+p^*(\left.\beta\right|_D)+i\pp\bp U_{n}\right)^{n-k}}{\omega_{\mathcal{C}}^n}.
	\end{align*}
	By Proposition \ref{solve poisson equation}, we know that $|z^\frac{k}{2}\nabla^k F_{n+1}(z,\cdot)|_{\omega_\Ca}\leqslant C_Kz^{-n-1+\epsilon}$, which is of the same order of $F_{n}$. Then 
	\begin{align*}
		F_{n+2}&=-\frac{ni\pp\bp u_{n+2}\wedge\sum_{k=1}^{n-1}{n-1 \choose k}(p^*(\beta|_D)+i\pp\bp U_{n+1})^k\wedge\omega_{\mathcal{C}}^{n-k-1}}{\omega_{\mathcal{C}}^n}\\
        &\qquad\qquad\qquad-\frac{\sum_{k=2}^{n}{n \choose k}\left(i\pp\bp u_{n+2}\right)^k\wedge\left(\omega_\mathcal{C}+p^*(\beta|_D)+i\pp\bp U_{n+1}\right)^{n-k}}{\omega_{\mathcal{C}}^n}.
	\end{align*} 
	We know from the estimate in Proposition \ref{solve poisson equation} that $\left|z^\frac{k}{2}\nabla^k u_{n+2}\right|\leqslant Cz^{-n+\epsilon}$. So $F_{n+2}\leqslant C z^{-n-2+\epsilon}$. 
 
    Let $U = \tilde{U}+ \chi_2(\Phi^{-1})^*u_{n+2}$, we can choose $K'$ large enough such that $\beta+i\pp\bp U$ is K\"ahler on $X$. Also we see from the construction of $U$ that 
    \begin{align}\label{equivalence of metric}
        \left|\,\omega_\mathcal{C}-\Phi^*(\beta+i\pp\bp U)\right|_{\omega_\Ca} = \left|\Phi^*(\beta)\right|_{\omega_\Ca}+\sum_{k=1}^{n+2}\left|i\pp\bp u_k\right|_{\omega_\Ca}\leqslant C z^{-1}.
    \end{align}
    Fix a point $p\in X$. Let $r(x)$ denote the distance function to $p$ with the metric $\beta+i\pp\bp U$. With this asymptotic behavior, we know that $r(x)$ is in the same order of the distance function on $\Ca$.
    Then outside a compact set on $X$ we have the estimate that 
    \begin{align*}
        &\left|1-\frac{(\beta+i\pp\bp U)^n}{\Omega_X\wedge\ols\Omega_X}\right| = (\Phi^{-1})^*\left(\left|1-\frac{(p^*(\left.\beta\right|_D)+ \omega_\Ca +i\pp\bp U_{n+2})^n}{\omega_\Ca^n}+O(e^{-\delta z^n})\right|\right)\\
        =&(\Phi^{-1})^*\left(|F_{n+2}|+O(e^{-\delta z^n})\right)\leqslant C(\Phi^{-1})^*\left(z^{-n-2+\epsilon}\right)\leqslant C r^{-2-\epsilon}.
    \end{align*}
    For the $\Omega_X$-compatible condition, we notice that the small term $u_{n+2}$ does not affect the integration of the form $\beta+i\pp\bp U$:
	\begin{align*}
	&\int_X (\beta+i\pp\bp U)^n-\Omega_X\wedge\ols\Omega_X = \int_X (\beta+i\pp\bp U)^n-(\beta+i\pp\bp \tilde{U})^n \\ 
	&=\lim_{\varepsilon\to 0}\int_{\{(\Phi^{-1})^*t\leqslant -\log \varepsilon\}} i\pp\bp(\rho_{A_3}((\Phi^{-1})^*u_{n+2}))\wedge \left(\sum_{k=0}^{n-1}{n\choose k}(\beta+i\pp\bp \tilde{U})^k\wedge i\pp\bp\left(\rho_{A_2}((\Phi^{-1})^*u_{n+2})\right)^{n-k-1}\right)\\
	& =\lim_{\varepsilon\to 0}\int_{\{t = -\log \varepsilon\}}d^c u_{n+2} \wedge \frac{1}{nz^{n-1}}\left(\sum_{k=0}^{n-1}{n\choose k}(\beta+\omega_\mathcal{C}+i\pp\bp U_{n+1})^k\wedge (i\pp\bp u_{n+2})^{n-k-1} \right)=0.
	\end{align*}
	So $\beta+i\pp\bp U$ is a K\"ahler form satisfying both $\Omega_X$-compatible condition and decay condition in Tian-Yau-Hein's package.
\end{proof}

\section{Existence and the proof}\label{behavior of beta and main theorem}

Now we are ready to apply Tian-Yau-Hein's package to deform our metric $\beta+i\pp\bp U$ to a Calabi-Yau metric.
\begin{theorem}
	 For any class $\mathfrak{k}$ in $H_{+}^2(X)$, there exists a Calabi-Yau metric $\omega$ in the class $\mathfrak{k}$.
\end{theorem}
\begin{proof}	
Let $\beta$ be the good representative we chose in $U$ be the potential constructed in section \ref{integral condition} with the form $\beta$. We know that $\beta+i\pp\bp U$ is a K\"ahler metric on $X$ such that 
\begin{align*}
    F(\beta+i\pp\bp U) := 1-\frac{(\beta+i\pp\bp U)^n}{\Omega_X\wedge\ols\Omega_X} 
\end{align*}
decays in order $r^{-2-\epsilon}$ and $\int_X (\beta+i\pp\bp U)^n $, here $r$ is any distance function under the metric $\beta+i\pp\bp U$. Let 
\begin{align*}
    f=\log{\frac{\Omega_X\wedge\ols\Omega_X}{(\beta+i\pp\bp U)^n}}.
\end{align*}
We have $f$ satisfies integral condition $\int_X (e^f-1)(\beta+i\pp\bp U)^n = 0$ and the decay condition $\left\vert f\right\vert\leqslant Cr^{-2-\epsilon}$.

On the other hand, we have higher regularity estimate of $u_i$'s:
\begin{align*}
    |z^{\frac{k}{2}}\nabla^k(\Phi^*(\beta+i\pp\bp U)-\omega_\mathcal{C})|_{\omega_\Ca} = |z^{\frac{k}{2}}\nabla^k\Phi^*(\beta)+\sum_{j=1}^{n+2}z^{\frac{k}{2}}\nabla^ki\pp\bp u_j|_{\omega_\Ca} \leqslant C z^{-1}
\end{align*} for any $z>C$ with some $C>0$. Then we have higher estimate of metric and scalar curvature. So $(X,\beta+i\pp\bp U)$ satisfies the $\mathrm{SOB}(\tfrac{2n}{n+1})$ condition by Lemma \ref{model space SOB} and $\mathrm{HMG}(\tfrac{1}{n+1},k,\alpha)$ by Lemma \ref{lemma of C^k,alpha} for any $k>0$ and $0<\alpha<1$.

So we know that there exists a function $\phi$ on $X$ such that 
\begin{align}\label{MA equation for phi}
	(\beta+i\pp\bp U +i\pp\bp \phi)^n = e^f(\beta+i\pp\bp U)^n= \Omega_X\wedge\ols\Omega_X,
\end{align} with $\phi \in C^{4}(X)$.
\end{proof}

The iteration process shows that for any $K>0$, there exists function $U_K$ and constant $C_K$ such that \begin{align*}
\left|f_K\right|:=|\,\log{\frac{\Omega_X\wedge\ols\Omega_X}{(\beta+i\pp\bp U_K)^n}}|\leqslant C_K r^{-K}.
\end{align*} If we choose $U$ such that the $\omega_\Ca$ potential $F(\beta+i\pp\bp U)$ decays fast enough, we can show that the solution $\phi$ of (\ref{MA equation for phi}) also decays fast to a constant. To do this, we first present the following local Poincar\'e lemma for $\mathrm{SOB}(\nu)$ manifold with $\nu \in (0, 2]$:
\begin{lemma}\label{local Poincare lemma}
 	Assume $(M, \omega_0)$ is a complete K\"ahler manifold satisfying $SOB(\nu)$ condition with $\nu \in (0, 2]$, and $r(x)^\kappa|B(x, 1)|\leqslant C$ as $r(x)\to \infty $ for some fixed $\kappa > 0$, $C > 0$. Let $u, f \in C^{\infty}(M)$ such that $\sup |\nabla^i u|+\sup |\nabla^i f|<\infty$ for all $i \in \mathbb{N}_0$, and $(\omega_0+i \partial \bar{\partial} u)^m=e^f \omega_0^m$. Then for any $\delta>0$, there exists $K_{\delta}>0$ such that if $\int_X |\nabla u|^2\omega^n<\infty$ and $|f|\leqslant C r^{-K_{\delta}}$, then $$\sup_{B(x,1)} |u-u_{B(x,1)}|\leqslant Cr(x)^{-\delta}$$  for any $x\in X$.
\end{lemma}

\begin{remark}
	The proof is entirely same as the proof in \cite[Proposition 4.8(ib)]{HeinThesis}. The only difference is that we choose $r_i$ to be $i$. So we omit the proof here.
\end{remark}

Then we can improve the $C^0$ bound of our solution $\phi$ to get the optimal close rate of our weak asymptotically Calabi metric:

\begin{theorem}
	For any class $\mathfrak{k}$ in $H_{+}^2(X)$, there is a Calabi-Yau metric $\omega$ in $\mathfrak{k}$ which is weak asymptotically Calabi with rate $1$.
\end{theorem}
\begin{proof}
With Lemma \ref{local Poincare lemma}, together with \cite[Proposition 4.8(ii)]{HeinThesis}, we know that if we choose $K$ large enough, there exists a constant $\bar{\phi}$ and $C$ such that the solution $\phi$ satisfies that
\begin{align*}
	|\phi-\bar{\phi}|\leqslant Cr^{-\delta+\frac{n}{n+1}}, \text{ for any } x \text{ such that } r(x) >C.
\end{align*}
Then we can replace $\phi$ by $\phi-\bar{\phi}$ to get a better candidate for the solution of (\ref{MA equation for phi}), so $\phi$ could be chosen to decay at any polynomial rate. Repeat our local rescaling and local Schauder estimate, we know that the Calabi-Yau metric $\beta+i\pp\bp U_K+i\pp\bp \phi$ is polynomially closed to the Calabi model space with the leading error term $\beta+i\pp\bp U_K$. 

If $\beta|_D =0$, the error term is exponentially close to Calabi model space. If $\beta|_D$ is nonzero, the decay rate of $\beta$ is exactly $r^{-\frac{2}{n+1}}$. If we choose $\beta$ such that $\beta|_D$ is primitive with respect to $\omega_D$, the decay of $i\pp\bp U_K$ would be $r^{-\frac{4}{n+1}+\epsilon}$, which is strictly lower order term compared with $\beta$. Thus, the Calabi-Yau metric $\beta+i\pp\bp (U+\phi)$ decays exactly at the rate $r^{-\frac{2}{n+1}}$, which is equivalent to $z^{-1}$.
\end{proof}

\section{Uniqueness}\label{Uniqueness}

In this section, we prove that the Calabi-Yau metric asymptotic to $\omega_\Ca$ in the class $\mathfrak{k}$ is unique.

\begin{theorem}\label{unique theorem 2}
	Let $(M,D)$ be the pair we considered before. If we have another Calabi-Yau metric $\tilde{\omega}$ in the same class $\mathfrak{k}$ satisfying $\left|\,\tilde{\omega}-\omega\right|_{\omega}\leqslant r^{-\kappa}$, when $r\to \infty$, for some distance function $r$ with respect to $\omega$ and some $\kappa>0$, then $\tilde{\omega} = \omega$.	
\end{theorem}

\begin{remark}
	We are also interested in the problem that how different choice of the diffeomorphism $\Phi$ will change our Calabi-Yau metric. For example, the scaling in the fiber direction will change the metric by the rate $r^{-\frac{2n}{n+1}}$ and by our uniqueness theorem, we get the same Calabi-Yau metric.
\end{remark}

The proof of the theorem can be sketched as follows. We start with a $\pp\bp$-lemma by solving $\bp$ equation via the $L^2$ method. Then we can write $\tilde{\omega} = \omega +i\pp\bp l$ with some estimate on $l$. By pulling back to $\mathcal{C}$, we construct $f$ on the model space to solve the Poisson equation $\Delta_\mathcal{\omega_\mathcal{C}} f = \Delta_\mathcal{\omega_\mathcal{C}}l$. Via the estimate of harmonic function on $\mathcal{C}$ in Sun-Zhang \cite{SZ}, we can use the equation $(\omega+i\pp\bp l)^n = \omega^n$ and take integration by parts to deduce that $i\pp\bp l = 0$.


\begin{lemma}\label{C0 bound of l}
    There exists a smooth function $l$ on $X$ such that $\tilde{\omega} = \omega +i\pp\bp l$ with $|\Phi ^*l|<Ce^{\epsilon t}$ on $\{t\geqslant C\}$ for any $\epsilon>0$ and some $C>0$.
\end{lemma}
\begin{proof}
    We prove the lemma by several steps:
    
    \textbf{Step 1:} We show that there exists a smooth $1$-form $\sigma$ on $X$ such that $\tilde{\omega} - \omega = d\sigma$ with $\left|\sigma\right|_{\omega}\leqslant Cz^{n+\frac{1}{2}-\kappa}$. 

     After pulling back to the model space $\Ca$ we have $\Phi^* (\tilde{\omega} - \omega)$ is a closed 2-form with $|\Phi^*(\tilde{\omega} - \omega)|_{\omega_\Ca}\leqslant Cz^{-\kappa}$. By viewing $\Ca$ as $Y\times (0,+\infty)$, we can write it as 
    \begin{align*}
        \Phi^* (\tilde{\omega} - \omega) = \eta + dz\wedge \gamma
    \end{align*} with $\pp_z\righthalfcup \eta = 0$, $\pp_z\righthalfcup \gamma = 0$. Then the fact that $d(\eta + dz\wedge \gamma) = 0$ implies $d_Y\eta = 0$, $\pp_z \eta = d_Y\gamma$.
    So we can choose \begin{align*}
    \tilde\sigma = \int_1^z \gamma dz\end{align*} such that \begin{align*}d\tilde\sigma =\eta+dz\wedge\gamma = \Phi^*(\tilde\omega-\omega).\end{align*} 
    Since $|\Phi^*(\tilde{\omega} - \omega)|_{\omega_\Ca}\leqslant Cz^{-\kappa}$, we have the decay of $dz\wedge\gamma$ which implies that \begin{align*}|\gamma|_{\omega_\Ca} \leqslant Cz^{\frac{n-1}{2}-\kappa}.\end{align*}
    Given the formula of $\omega_\Ca$ we can have an estimate of $|\tilde\sigma|_{\omega_\Ca}$ at the point $(y,z_0)\in \Ca$:
    \begin{align*}
        |\tilde\sigma(y,z_0)|_{\omega_\Ca(y,z_0)} &\leqslant \int_1^{z_0} |\gamma(y,z)|_{\omega_\Ca(y,z_0)} dz\leqslant \int_1^{z_0} |\gamma(y,z)|_{\omega_\Ca(y,z)} z_0^{\frac{n-1}{2}} z^{\frac{1}{2}} dz \leqslant C z_0^{n+\frac{1}{2}-\kappa}.
    \end{align*}
    After extending $(\Phi^{-1})^*\tilde\sigma$ as a smooth 1-form on $X$, we can write $\tilde\omega -\omega = d((\Phi^{-1})^*\tilde\sigma) + \theta$ for some smooth compact supported closed 2-form $\theta$ on X. 
    
    Recall that $X$ is 1-convex. Then by the vanishing theorem for 1-convex manifold from Van Coevering \cite{van2008construction} Proposition 4.2., $\theta = i\pp\bp s = dd^cs$ for some compact supported function $s$ on $X$. Then $\sigma = (\Phi^{-1})^*\tilde\sigma + d^cs$ is the smooth 1-form that we are looking for.
 
\textbf{Step 2:} Recall that $E$ is the non-ample locus of $-K_M$. The $X\setminus E$ admits a complete K\"ahler metric by Proposition 4.1 in Ohsawa \cite{Ohsawa1984VanishingTO}. So we can use $L^2$-estimate on $X\setminus E$ to solve the $\bp$ equation to construct the potential $l$ such that $\tilde\omega -\omega  = i\pp\bp l$.

    Let $\tau= \epsilon\cdot\rho_{B_1}(\mathfrak{t})-\delta\cdot\rho_{B_2}(\mathfrak{t})^{\frac{1}{n}}$. Choose $z_0$, $B_1$ and $B_2$ large, then choose $\delta$ small depending on $\epsilon$, we can guarantee that the $(1,1)$ form \begin{align*}
        i\pp\bp \tau = i\pp\bp (\epsilon\cdot\rho_{B_1}(\mathfrak{t})-\delta\cdot\rho_{B_2}(\mathfrak{t})^{\frac{1}{n}})
    \end{align*} is a K\"ahler form on $X\setminus E$. We have $i\pp\bp \tau \geqslant C_{\epsilon,\delta} \mathfrak{t}^{-1} \omega$ outside a compact set.
    
    If we take the type decomposition of $\left(\Phi^{-1}\right)^*\tilde\sigma = \left(\left(\Phi^{-1}\right)^*\tilde\sigma\right)^{1,0}+\left(\left(\Phi^{-1}\right)^*\tilde\sigma\right)^{0,1}$, we have the estimate of $\left(\left(\Phi^{-1}\right)^*\tilde\sigma\right)^{0,1}$ that $\left|\left(\left(\Phi^{-1}\right)^*\tilde\sigma\right)^{0,1}\right|_\omega \leqslant C \mathfrak{t}^{1+\frac{1-2\kappa}{2n}}$ for $\mathfrak{t}>C$ and $\left(\Phi^{-1}\right)^*\tilde\sigma$ supported on $\mathfrak{t}>C$.
    
    So with the same weighted $L^2$ estimate in Hein-Sun-Viaclovsky-Zhang \cite[Proposition 2.2.]{hsvz2}, we have 
    \begin{align*}
        \int_{X\setminus E} \mathfrak{t}\cdot\left|(\Phi^{-1})^*\tilde\sigma^{0,1}\right|_{\omega} e^{-\tau} \omega^n\leqslant C\int_{X\setminus E} \mathfrak{t}\cdot\mathfrak{t}^{\frac{2n+1-2\kappa}{2n}} e^{-\epsilon\mathfrak{t}+\delta\mathfrak{t}^\frac{1}{n}} \omega^n < \infty
    \end{align*}
    which yields that we have a solution $\iota$ such that $\bp \iota = \sigma^{0,1}$ with
    \begin{align*}
        \int_{X\setminus E} |\iota|^2 e^{-\tau}\omega^n\leqslant\int_{X\setminus E} \mathfrak{t}\cdot|(\Phi^{-1})^*\tilde\sigma^{0,1}|_{\omega} e^{-\tau} \omega^n.
    \end{align*}
    Consequently, we have $i\pp\bp(2Im\iota)  = d(\Phi^{-1})^*\tilde\sigma$. Set $l = 2Im\iota+s$ then we have $\tilde\omega-\omega =  i\pp\bp l$.
    
\textbf{Step 3:} We give the $C^0$ bound and $C^k$ bound for $l$ via elliptic estimates on the scaled metric.

    Let $x$ be any point in $X\setminus K$. With the same local elliptic estimate under the scaled metric $\hat\omega = \mathfrak{t}(x)^{-\frac{1}{n}}\omega$ as in Proposition \ref{solve poisson equation}, we can give a global $C^0$ bound of $l$. We know that $l$ satisfies the elliptic equation $(\omega+i\pp\bp l)^n = \omega^n$ 
    with 
    \begin{align*}
        \int_{B_{\hat\omega}(x,1)}|l|^2 \omega^n \leqslant e^{\epsilon \cdot C\mathfrak{t}(x)}\int_{X\setminus E} |l|^2 e^{-\epsilon\cdot\rho_{B_1}(\mathfrak{t})+\delta\cdot\rho_{B_2}(\mathfrak{t})^{\frac{1}{n}}}\omega^n\leqslant C_\epsilon e^{\epsilon \cdot C\mathfrak{t}(x)},
    \end{align*} 
    since we have some uniform constant $C$ such that $\mathfrak{t}(y)\leqslant C\cdot \mathfrak{t}(x)$ for any $y \in B_{\hat\omega}(x,1)$ and any $x\in X\setminus K$. By adjusting $\epsilon$ small enough we have
    \begin{align*}
        \|l\|_{L^2(B_{\hat\omega}(x,1))}\leqslant C_\epsilon e^{\epsilon \mathfrak{t}(x)}.
    \end{align*}
    
    Now we can do local elliptic estimates on the scaled metric after lifting to the universal cover. 
    Since the $S^1$ direction on $\Ca$ collapsing in polynomial order with respect to $z$, we know that 
    \begin{align*}
        \|l\|_{L^2(\widetilde{B}_{\hat\omega}(x,1))}\leqslant C\cdot\mathfrak{t}(x)^{\frac{n-1}{2n}}\cdot\|l\|_{L^2(B_{\hat\omega}(x,1))}\leqslant C_\epsilon e^{\epsilon \mathfrak{t}(x)}.
    \end{align*}We have the global $C^0$ bound for $l$:
    \begin{align*}
        |l|(x)\leqslant\|l\|_{W^{2,2}(\widetilde{B}_{\hat\omega}(x,1))}\leqslant C\cdot\|l\|_{L^2(\widetilde{B}_{\hat\omega}(x,1))}\leqslant C_\epsilon e^{\epsilon \mathfrak{t}(x)}
    \end{align*}
    for any $\epsilon>0$.
\end{proof}
\begin{remark}
In the proof of the $i\pp\bp$-lemma \ref{C0 bound of l} we did not use the polynomial decay of $\tilde\omega-\omega$. In fact, we can always find $l$ even when $|\tilde\omega-\omega|$ is polynomially growth.
\end{remark}
Furthermore, we can prove that $\tilde\omega-\omega$ has weighted higher regularity bound.
\begin{lemma}\label{higher regularity of dsigma}
    There exists a constant $C>0$ such that \begin{align*}
        \left|z^{\frac{k}{2}}\nabla^k (\tilde\omega-\omega)\right|_{\omega}\leqslant Cz^{-\kappa}
    \end{align*} for any $z>C$.
\end{lemma}
\begin{proof}
    Fix any point $x$ in $X$ with $\mathfrak{t}(x) = z_0^n$. We still work on the scaled metric $\hat\omega=z_0^{-1}\omega$ with uniform bounded curvature. The injectivity radius of the universal covering around $x$ is bounded below by a universal constant $\delta$ independent of $x$. 
    Now we are working on the ball $\tilde B(\tilde{x},\delta)$ in the universal cover. Since $\tilde\omega-\omega$ is $d$-exact, locally we can take integration of $\tilde\omega-\omega$ along the geodesic lines to have 1-form $\sigma$ on $\tilde B(\tilde{x},\delta)$ such that  
    \begin{align*}
        d\sigma =\tilde\omega-\omega,\qquad \|\sigma\|_{C^0_{\hat\omega}(\tilde B(\tilde{x},\delta))}\leqslant \|\tilde\omega-\omega\|_{C^0_{\hat\omega}(\tilde B(\tilde{x},\delta))}\leqslant C z_0^{1-\kappa}.
    \end{align*}
    Consider the type decomposition of $\sigma = \sigma^{0,1}+\sigma^{1,0}$. The operator $$N:L^2(\tilde B(\tilde{x},\delta),\Omega^{0,1})\to L^2(\tilde B(\tilde{x},\delta),\Omega^{0,1})$$
    constructed in \cite[Theorem 8.9]{kohn1963harmonic} satisfies that 
    \begin{align*}
        \Delta_{\bp}(N\sigma^{0,1})= \sigma^{0,1},\quad \|N\sigma^{0,1}\|_{L^2_{\hat\omega}(\tilde B(\tilde{x},\delta))}\leqslant C \|\sigma^{0,1}\|_{L^2_{\hat\omega}(\tilde B(\tilde{x},\delta))}
    \end{align*} and $N$ commutes with $\pp$ and $\bp$.
    Then
    \begin{align*}
        \|N\sigma^{0,1}\|_{W^{2,2}_{\hat\omega}(\tilde B(\tilde{x},\delta))}\leqslant C \|\sigma^{0,1}\|_{L^2_{\hat\omega}(\tilde B(\tilde{x},\delta))}\leqslant C \|\sigma^{0,1}\|_{C^0_{\hat\omega}(\tilde B(\tilde{x},\delta))}.
    \end{align*}
    Then we know by Sobolev lemma and iteration process that for any $q>1$
    \begin{align*}
        \|N\sigma^{0,1}\|_{W^{2,q}_{\hat\omega}(\tilde B(\tilde{x},\delta))}\leqslant C \|\sigma^{0,1}\|_{L^q_{\hat\omega}(\tilde B(\tilde{x},\delta))}\leqslant C \|\sigma^{0,1}\|_{C^0_{\hat\omega}(\tilde B(\tilde{x},\delta))}.
    \end{align*}
    Take $q>n$, there exists $\alpha>0$ such that
    \begin{align*}
       \|N\sigma^{0,1}\|_{C^{1,\alpha}_{\hat\omega}(\tilde B(\tilde{x},\delta))}\leqslant C\|N\sigma^{0,1}\|_{W^{2,q}_{\hat\omega}(\tilde B(\tilde{x},\delta))}\leqslant C \|\sigma^{0,1}\|_{C^0_{\hat\omega}(\tilde B(\tilde{x},\delta))}.
    \end{align*}
    Let $f= \bp^*N\sigma^{0,1}$. We have 
    \begin{align*}
       \|f\|_{C^{0,\alpha}_{\hat\omega}(\tilde B(\tilde{x},\delta))}\leqslant C \|\sigma^{0,1}\|_{C^0_{\hat\omega}(\tilde B(\tilde{x},\delta))},\qquad\bp f = \Delta_{\bp}(N\sigma^{0,1})=\sigma^{0,1}.
    \end{align*}
    Let $\hat{l} = z_0^{-1}\cdot 2\mathrm{Im}f$. We have $i\pp\bp \hat{l} = z_0^{-1}(\tilde\omega-\omega)$. Hence $(\hat\omega+i\pp\bp \hat{l})^n = \hat\omega^n $ with \begin{align*}
        \|\hat{l}\|_{C^{0,\alpha}_{\hat\omega}(\tilde B(\tilde{x},\delta))}\leqslant C z_0^{-\kappa}.
    \end{align*} By Schauder estimates we have higher regularity \begin{align*}
        \|\hat{l}\|_{C^{k,\alpha}_{\hat\omega}(\tilde B(\tilde{x},\delta))}\leqslant C z_0^{-\kappa} 
    \end{align*}
    which yields \begin{align*}
        \left|z_0^{\frac{k}{2}}\nabla^k (\tilde\omega-\omega)\right|_{\omega}\leqslant Cz_0^{-\kappa}
    \end{align*} for any $z_0>C$.
\end{proof}

Then we are ready to prove the uniqueness:
\begin{proof}[Proof of Theorem \ref{unique theorem 2}]
Given by previous estimate, we have $\tilde \omega-\omega = i\pp\bp l$ with $|l|\leqslant C_\epsilon e^{\epsilon \mathfrak{t}}$. If we pull back $l$ to $\mathcal{C}$, by the closeness of complex structure we have \begin{align*}
	|dJ_\Ca d l|_{\omega_\Ca} \leqslant |d(J_\Ca-J_X) d l|_{\omega_\Ca} + |dJ_X d l|_{\omega_\Ca}\leqslant C(e^{-(\frac{1}{2}-\epsilon)z^n}+z^{-\kappa}).
\end{align*}
The function $F_l = \Delta_{\omega_\Ca}l$ has higher regularity bound on $\Ca$:
\begin{align*}
	\left|z^{\frac{k}{2}}\nabla^k F_l\right|_{\omega_{\Ca}}\leqslant Cz^{-\kappa} \text{ for any } z>C. 
\end{align*}By Proposition \ref{solve poisson equation} there exists a smooth function
$f$ on $\mathcal{C}$ such that $\Delta_{\omega_\mathcal{C}}f = F_l$ with
\begin{align*}
    |i\pp\bp f|_{\omega_\mathcal{C}}\leqslant Cz^{-\kappa+\epsilon}, \quad|df|_{\omega_\Ca}\leqslant Cz^{\frac{n+1}{2}-\kappa+\epsilon}, \quad |f|\leqslant C z^{n+1-\kappa+\epsilon} \text{ for any }\epsilon>0.
\end{align*}
Since $\Delta_{\omega_\Ca}(l-f) =0$ and $|l-f|\leqslant e^{\epsilon t}$ for any $\epsilon>0$, from the behavior of harmonic function \cite[Proposition 5.3.]{SZ} we know that $l= f+ \lambda z+ g+ O(e^{-\delta z})$ for some $\lambda>0$ and some harmonic $S^1$-invariant function $g$ on $\Ca$ with $|g|\leqslant C e^{\delta z^{\frac{n}{2}}}$. Since $|z^{n-1}g_{tt}(t,q)|\leqslant |i\pp\bp g|_{\omega_\Ca} \leqslant C z^{-\kappa+\epsilon}$ holds uniformly for any $q\in D$, integration along the $\mathbb{R}$-fiber direction shows that $g$ and hence $l$ is at most polynomially growth. Again by \cite[Proposition 5.3.]{SZ} we know that $l = f + \lambda z + O(e^{-\delta z})$.

Recall that $l$ satisfies that $$\Delta_{\omega_{\Ca}} l = \Delta_{\omega_{\Ca}} l - \Delta_{\Phi^*\omega} l + \sum_{k = 2}^n {n\choose k}(d\Phi^*J_Xd l)^k\wedge\Phi^*\omega^{n-k},$$ by our previous construction we know that $|\Phi^*\omega -\omega_\Ca|_{\omega_\Ca}\leqslant Cz^{-1}$, so $|\Delta_{\omega_{\Ca}} l - \Delta_{\Phi^*\omega} l|_{\omega_\Ca}\leqslant Cz^{-1-\kappa}$,
 \begin{align}\label{uniqueness iteration}
 	\Delta_{\omega_{\Ca}} l \leqslant C\left(z^{-1}|i\pp\bp l|_{\omega_\Ca}+|i\pp\bp l|^2_{\omega_\Ca}\right) \leqslant Cz^{-\min\{1+\kappa,2\kappa,1+n\}}.
 \end{align}
Thus by finite step iteration we can find a better candidate $\tilde f$ and another constant $\tilde\lambda$ such that 
\begin{align*}
	l =  \tilde f+\tilde\lambda z+O(e^{-\delta z}),
	\quad|\nabla^2 \tilde f|_{\omega_\mathcal{C}}\leqslant C z^{-n-1+\epsilon},
	\quad |\nabla \tilde f|_{\omega_\Ca}\leqslant C z^{-\frac{n+1}{2}+\epsilon}, 
	\quad |\tilde f|<Cz^\epsilon.
\end{align*}

On the other hand,
\begin{align*}
	0&= \int_X i\pp\bp l\wedge \sum_{k=1}^n{n\choose k}(i\pp\bp l)^{k-1}\wedge \omega^{n-k}
	=\lim_{\varepsilon\to 0}\int_{\{(\Phi^{-1})^*t\leqslant -\log \varepsilon\}} i\pp\bp l\wedge \sum_{k=1}^n{n\choose k}(i\pp\bp l)^{k-1}\wedge \omega^{n-k}\\
	&=\lim_{\varepsilon\to 0}\int_{\{t=-\log \varepsilon\}} \tilde\lambda d^cz\wedge \sum_{k=1}^n{n\choose k}(i\pp\bp l)^{k-1}\wedge \omega^{n-k}
	=\lim_{\varepsilon\to 0}\int_{\{t=-\log \varepsilon\}} \tilde\lambda d^ct\wedge (i\pp\bp t)^{n-1}=\tilde\lambda \mathrm{Vol}(D).
\end{align*}

Consequently, $\tilde\lambda =0$, $l \leqslant Cz^\epsilon$ for any $\epsilon>0$. From the equation of $l$ we know that
\begin{align*}
	&\lim_{\varepsilon\to 0}\left|\int_{\{\left(\Phi^{-1}\right)^*t=-\log \varepsilon\}} ld^cl\wedge \sum_{k=1}^n{n\choose k}(i\pp\bp l)^{k-1}\wedge \omega^{n-k}\right|
	\leqslant \lim_{\varepsilon\to 0}\int_{\{t=-\log \varepsilon\}} \left|\frac{C}{z^{1-\epsilon}}\cdot d^ct \wedge(i\pp\bp t)^{n-1}\right|=0.
\end{align*}
Hence by integration by parts and $0=l(\tilde\omega^n-\omega^n)=l\cdot i\pp\bp l\wedge\sum_{k=0}^{n-1}\omega^k\wedge\tilde\omega^{n-1-k}$:
\begin{align*}
	0=-\int_X l\cdot i\pp\bp l\wedge\sum_{k=0}^{n-1}\omega^k\wedge\tilde\omega^{n-1-k}
	=\int_X dl\wedge d^cl\wedge\sum_{k=0}^{n-1}\omega^k\wedge\tilde\omega^{n-1-k}.
\end{align*}

Since $\sum_{k=0}^{n-1}\omega^k\wedge\tilde\omega^{n-1-k}$ is a positive form, we know that $dl=d^cl =0$.
\end{proof}

\section{Discussion and Questions}\label{discussions}

\subsection{Examples}
We present examples that $(X,\omega)$ is a Calabi-Yau manifold not asymptotically Calabi but weak asymptotically Calabi under the fixed diffeomorphism $\Phi$. As discussed in the end of the proof of Theorem \ref{main theorem 1}, we have the following:
\begin{claim}
	Let $(M,D)$ be the pair in Definition \ref{setting} with $X = M\setminus D$. Let $H_{+,c}^2(X) = \mathrm{Im}(H^2_c(X)\to H^2(X))\cap H_{+}^2(X)$. Fix a diffeomorphism $\Phi:\Ca\setminus \mathcal{K}\to X\setminus K$. Then for any $\mathfrak{k}$ in $H^2_+(X)$ but not $H^2_{+,c}(X)$ the metric $\omega$ we constructed in Theorem \ref{main theorem 1} is a Calabi-Yau metric not asymptotically Calabi but weak asymptotically Calabi.
\end{claim}
\begin{example}
	Let $M=\dP^1\times \dP^2$ with two projection maps $\pi_1:M\to \dP^1$ and $\pi_2:M\to \dP^2$. Then we have $D =-K_M= \pi_1^*O_{\dP^1}(2)\otimes \pi_2^*O_{\dP^2}(3)$. $\mathrm{Pic}(M)$ is generated by $\pi_1^*O_{\dP^1}(1)$ and $\pi_2^*O_{\dP^2}(1)$ and the image of each of them under the map $i^*: H^2(M)\to H^2(D)$ induced by the inclusion map $i:D\to M$ is not parallel to $[\omega_D] = c_1(N_D)$. Choose a primitive representative of any of these two classes and apply Theorem \ref{main theorem 1} we will find a Calabi-Yau metric not asymptotically Calabi but weak asymptotically Calabi.
\end{example}

These kind of examples could be found on any Fano manifold $M$ with $\mathrm{dim}_\mathbb{C} M\geqslant 3$ and $h_2(M)\geqslant 2$. We can find many examples in Mori-Mukai \cite{mori3fold}. Besides, there are also many examples in the weak Fano case but we do not have a simple topological sufficient condition.

\subsection{Weaker Decay Condition}
In our statement of uniqueness Theorem \ref{unique theorem 1}, we need the metric $\tilde\omega$ to be polynomially closed to $\omega$. The main difficulty to get rid of this condition lies in how to deduce the decomposition of $l = f+\lambda z+ O(e^{-\delta z})$ with $|f|\leqslant z^\epsilon$ for any $\epsilon>0$, where we cannot do iteration to improve the decay of $f$ as in (\ref{uniqueness iteration}).

It is natural to ask the following question:
\begin{question}
	Can we prove a stronger uniqueness theorem: If we have another Calabi-Yau metric $\tilde\omega$ such that $|\tilde\omega-\omega|_\omega\to 0$ when $r\to \infty$ for some distance function $r$ with respect to $\omega$, then $\tilde\omega=\omega$? 
	\end{question}
	
	One possible obstruction of this stronger uniqueness theorem is that we cannot rule out the possibility that there is a Calabi-Yau metric $\omega$ closed to the Calabi model space in a logarithm rate rather than any polynomial rate. The existence of this type of Calabi-Yau metric is also an interesting question to study.
	
\subsection{Compactification and Classification}


Hein-Sun-Viaclovsky-Zhang \cite{hsvz2} showed that any asymptotically Calabi manifold which is Calabi-Yau can be compactified complex analytically to a weak Fano manifold and the Calabi-Yau comes from the construction by Tian-Yau-Hein's package. 

In the weak asymptotic Calabi manifold case, when we only have the exponential closeness of complex structure, even though the metric is polynomial close, we can still get that any weak asymptotically Calabi manifold which is Calabi-Yau can be compactified complex analytically to a weak Fano manifold by repeating the argument in \cite{hsvz2} as one can also do $L^2$ estimate to construct the holomorphic function on $X$ from the holomorphic section of $N_D$. The key difference is to show that the compactification we get is K\"ahler by considering the behavior of the class at the end. By our uniqueness Theorem \ref{unique theorem 1} this Calabi-Yau metric $\omega$ comes from our generalized Tian-Yau construction in Theorem \ref{main theorem 1}.

We would like to make the following conjecture to further generalize this into slower decay assumption.
\begin{definition}
Let $X$ be a complete K\"ahler manifold with complex dimension $n$, complex structure $I$, K\"ahler form $\omega$ and $(n,0)$-form $\Omega$. We say $(X, I, \omega, \Omega)$ is \emph{polynomial asymptotically Calabi with rate} $(\kappa_1, \kappa_2)$ if:

 there exists $\kappa_1,\kappa_2>0$, a Calabi model space $(\mathcal{C}, I_\Ca, \omega_\Ca,\Omega_\Ca)$, and a diffeomorphism $\Phi: \mathcal{C} \setminus \mathcal{K} \rightarrow X \setminus K$, where $K \subset X$ and $\mathcal{K}\subset \mathcal{C}$ are compact, such that the following hold uniformly as $z \rightarrow+\infty$ :
	$$\left|\nabla_{\omega_{\mathcal{C}}}^k\left(\Phi^* I_X-I_{\mathcal{C}}\right)\right|_{\omega_{\mathcal{C}}} +\left|\nabla_{\omega_\Ca}^k\left(\Phi^* \Omega-\Omega_{\mathcal{C}}\right)\right|_{\omega_\Ca}=O\left(z^{-\kappa_1}\right),\quad\left|\nabla_{\omega_{\mathcal{C}}}^k\left(\Phi^* \omega-\omega_{\mathcal{C}}\right)\right|_{\omega_{\mathcal{C}}} =O\left(z^{-\kappa_2}\right)$$ for all $k \in \mathbb{N}_0$.
\end{definition}
\begin{conjecture}
	There are optimal constants $\lambda$ and $\mu$ such that for any $\kappa_1>\lambda$ and $\kappa_2>\mu$, any polynomial asymptotically Calabi Calabi-Yau manifold with rate $(\kappa_1, \kappa_2)$ can be compactified complex analytically to a weak Fano manifold. Furthermore, the Calabi-Yau metric comes from our generalized Tian-Yau construction in Theorem \ref{main theorem 1}.
\end{conjecture}

When $\kappa_1$ is large enough, we can still use $L^2$ estimate to construct holomorphic coordinate on the end of $\ols{X}$. However, the question to find optimal $\lambda$ may not be approachable by $L^2$ method. From the uniqueness theorem \ref{unique theorem 1} and the compactification process of weak asymptotically Calabi Calabi-Yau manifold, we expect that $\mu$ should be $0$.
 
\appendix
\section{estimate of the solution of ODE}\label{estimate of ODE}

In this section, we will look closely to the solution of the following ordinary differential equation:
\begin{align*}
    u''-(\frac{j^2n^2}{4}+n\lambda)z^{n-2} u = nz^{n-1}v,
\end{align*}
where $\lambda>0$, $n\geqslant 3$ and $n, j\in \mathbb{N}$.\\
By the transformation in \cite{SZ}, we have two cases: zero node case when $j=0$ and non-zero node case when $j>0$. We will give a brief summary of the estimate of fundamental solutions and have a estimate of $u$ with polynomial rate which slightly generalizes the results in \cite{SZ}.\\
\subsection{fundamental solution of zero mode}

In this section we focus on the zero mode: the equation 
\begin{align}\label{equation of zero mode}
    u''-n\lambda z^{n-2} u = nz^{n-1}v.
\end{align}
By \cite{SZ} we have the decay solution $\mathcal{D}(z)$ and growth solution $\mathcal{G}(z)$ of the homogeneous equation $u''(z) = nz^{n-2}\lambda u(z)$ given by
	\begin{align}\label{expression D and G 0}
	    \mathcal{D}(z) =\sqrt{z}K_{\frac{1}{n}}\left(2\sqrt{\tfrac{\lambda}{n}}\cdot z^{\frac{n}{2}}\right),\\
		\mathcal{G}(z) =\sqrt{z}I_{\frac{1}{n}}\left(2\sqrt{\tfrac{\lambda}{n}}\cdot z^{\frac{n}{2}}\right)
	\end{align}
	where $K$ and $I$ have the following expression: for $\nu \in \mathbb{R}$
	\begin{align*}
	    K_\nu(y)&=\int_0^{\infty} e^{-y \cosh t} \cosh (\nu t) dt,\\
        I_\nu(y)&=\frac{1}{\pi} \int_0^\pi e^{y \cos \theta} \cos (\nu \theta) d \theta-\frac{\sin (\nu \pi)}{\pi} \int_0^{\infty} e^{-y \cosh t-\nu t} d t
	\end{align*}
\begin{lemma}\cite{SZ}\label{estimate of I and K}[Proposition 3.3.]
     We have the following uniform estimate:\begin{enumerate}
         \item For all $\nu \in \mathbb{R}$, there is a constant $C(\nu)>1$ such that
            $$
            \begin{array}{l}
            C^{-1}(\nu) \cdot \frac{e^{-y}}{\sqrt{y}} \leqslant K_\nu(y) \leqslant C(\nu) \cdot \frac{e^{-y}}{\sqrt{y}}, \quad y \geqslant 1 ; \\
            I_\nu(y) \leqslant \begin{cases}C(\nu) \cdot \frac{e^y}{\sqrt{y}}, & y \geqslant 1, \\
            C(\nu) \cdot y^\nu, & 0<y \leqslant 1 .\end{cases}
            \end{array}
            $$
         \item   For all $\nu>-1$, we have
            $$
            I_\nu(y) \geqslant \begin{cases}C(\nu)^{-1} \cdot \frac{e^y}{\sqrt{y}}, & y \geqslant 1 \\ C(\nu)^{-1} \cdot y^\nu, & 0<y \leqslant 1\end{cases}
            $$
     \end{enumerate}
\end{lemma}

\begin{corollary}\cite{SZ}  
	For $z > \sqrt[\leftroot{-3}\uproot{3}n]{\frac{n}{4\lambda_1}}$, there exists a constant $C$ which only depends on $n$ such that 
	\begin{align*}
		\frac{1}{C}\cdot\frac{
		e^{-\frac{2}{\sqrt{n}}\lambda^{\frac{1}{2}}\cdot z^{\frac{n}{2}}}		}{
		\lambda^{\frac{1}{4}}\cdot z^{\frac{n-2}{4}}
		}<&\; \mathcal{D}(z)< C \cdot\frac{
		e^{-\frac{2}{\sqrt{n}}\lambda^{\frac{1}{2}}\cdot z^{\frac{n}{2}}}		}{
		\lambda^{\frac{1}{4}}\cdot z^{\frac{n-2}{4}}
		}\\
		\frac{1}{C}\cdot\frac{
		e^{\frac{2}{\sqrt{n}}\lambda^{\frac{1}{2}}\cdot z^{\frac{n}{2}}}		}{
		\lambda^{\frac{1}{4}}\cdot z^{\frac{n-2}{4}}
		}<&\; \mathcal{G}(z)< C \cdot\frac{
		e^{\frac{2}{\sqrt{n}}\lambda^{\frac{1}{2}}\cdot z^{\frac{n}{2}}}		}{
		\lambda^{\frac{1}{4}}\cdot z^{\frac{n-2}{4}}
		}.
	\end{align*}    
\end{corollary}
		
With those estimates, we can give a $C^0$ bound of $u(z)$. By computation in \cite{SZ} we know that the Wronskian $$\mathcal{W}(\mathcal{G},\mathcal{D}) =\mathcal{G}(z)\mathcal{D}'(z)-\mathcal{G}'(z)\mathcal{D}(z) = -\frac{n}{2}.$$ Hence we have a solution of \ref{equation of zero mode} as follows:
\begin{align}\label{u(D,G)}
		u(z) &= -2\left(\mathcal{D}(z)\int_1^z\mathcal{G}(s)s^{n-1}v(s)ds+\mathcal{G}(z)\int_z^\infty \mathcal{D}(s)s^{n-1}v(s)ds\right)
	\end{align}

We firstly introduce an estimate of the solution of this ordinary differential equation:
\begin{proposition}
    \label{estimate of u wrt v 0}
	Recall that $\lambda_1$ is the first nonzero positive eigenvalue of $-\Delta_Y$.
	Let $v$ be a function such that $|v(z)|\leqslant C_0 z^\delta$ for $z>1$. For any $\lambda$ such that $\lambda > \lambda_1>0$, we can find solution of equation $u''(z) = nz^{n-2}\lambda u(z) + nz^{n-1}v$ such that $|u(z)|\leqslant C \cdot C_0z^{\delta +1}$, $|u'(z)|\leqslant C\cdot C_0z^{\delta +\frac{n}{2}}$, $|u''(z)|\leqslant C\cdot C_0z^{\delta +n-1}$ on $z>C$ for some constant $C>1$ only depend on $n$, $\lambda_1$ and $\delta$ .
\end{proposition}
\begin{proof}
	
	Now we can estimate $\left\|u\right\|_{L^\infty((\max\{1,(\frac{n}{4\lambda_1})^{\frac{1}{n}}\},\infty))}$. Let $\mu  = \tfrac{2}{\sqrt{n}}\lambda^{\frac{1}{2}}$. By integration by parts the first term $\mathcal{D}(z)\int_1^z\mathcal{G}(s)s^{n-1}v(s)ds$ in \ref{u(D,G)} is bounded by a constant $C(n)$ times the following term:
	\begin{align*}
		&\frac{1}{\lambda^{\frac{1}{2}}\cdot z^{\frac{n-2}{4}}
		}\int_1^ze^{\,\mu\left(s^{\frac{n}{2}}-z^{\frac{n}{2}}\right)}s^{\delta+\frac{3n-2}{4}}ds\\
		 \leqslant &\,\frac{1}{\sqrt{n}\lambda}z^{\delta+1} -\frac{\delta+\frac{n+2}{4}}{n\lambda^{\frac{3}{2}}}z^{\delta+1-\frac{n}{2}} + \frac{(\delta+\frac{n+2}{4})(\delta+\frac{-n+2}{4})}{n\sqrt{n}\lambda^{2}}z^{\delta+1-n}\\
		 &\hspace{2.5cm}-\frac{(\delta+\frac{n+2}{4})(\delta+\frac{-n+2}{4})(\delta+\frac{-3n+2}{4})}{n^2\lambda^{\frac{5}{2}}}z^{\delta+1-\frac{3n}{2}}\\
		 &\hspace{2.5cm}+\frac{\max\{0,(\delta+\frac{n+2}{4})(\delta+\frac{-n+2}{4})(\delta+\frac{-3n+2}{4})(\delta+\frac{-5n+2}{4})\}}{n^2\lambda^{\frac{5}{2}}}z^{\delta+1-\frac{3n}{2}}\\
		 &\hspace{2.5cm}-\left(\frac{1}{\sqrt{n}\lambda} -\frac{(\delta+\frac{n+2}{4})}{n\lambda^{\frac{3}{2}}} +\frac{(\delta+\frac{n+2}{4})(\delta+\frac{-n+2}{4})}{n\sqrt{n}\lambda^{2}}+C(n,\delta)\frac{1}{\lambda^{\frac{5}{2}}}\right)e^{\mu}\frac{e^{-\mu z^{\frac{n}{2}}}}{z^{\frac{n-2}{4}}}.
	\end{align*}
Here we also use the following observation:
Since $\mu\geqslant \frac{2}{\sqrt{n}}\lambda_1^{\frac{1}{2}}$, we know that the maximum of $s^{\delta+\frac{-5n-2}{4}}e^{\,\mu\left(s^{\frac{n}{2}}-z^{\frac{n}{2}}\right)} $ on the interval $[1,z]$ is at $z$ when $z$ is larger than a uniform constant which is independent with respect to $\lambda$ but only on $n$ and $\lambda_1$. So we have 
$$\int_1^zs^{\delta+\frac{-5n-2}{4}}e^{\,\mu\left(s^{\frac{n}{2}}-z^{\frac{n}{2}}\right)}ds\leqslant (z-1)z^{\delta+\frac{-5n-2}{4}}<z^{\delta+\frac{-5n+2}{4}}.$$

For the second term $\mathcal{G}(z)\int_z^\infty \mathcal{D}(s)s^{n-1}v(s)ds$, we have similar estimate:
	\begin{align*}
		&\frac{1}{\lambda^{\frac{1}{2}}\cdot z^{\frac{n-2}{4}}
		}\int_z^\infty e^{\,\mu\left(z^{\frac{n}{2}}-s^{\frac{n}{2}}\right)}s^{\delta+\frac{3n-2}{4}}ds\\
		 \leqslant &\,\frac{1}{\sqrt{n}\lambda}z^{\delta+1} +\frac{\delta+\frac{n+2}{4}}{n\lambda^{\frac{3}{2}}}z^{\delta+1-\frac{n}{2}} + \frac{(\delta+\frac{n+2}{4})(\delta+\frac{-n+2}{4})}{n\sqrt{n}\lambda^{2}}z^{\delta+1-n}\\
		 &\hspace{2cm}+\frac{2\left|(\delta+\frac{n+2}{4})(\delta+\frac{-n+2}{4})(\delta+\frac{-3n+2}{4})\right|}{n^2\lambda^{\frac{5}{2}}}z^{\delta+1-\frac{3n}{2}}.
	\end{align*}

So we have the uniform estimate for $u$ that for any $z>C(n,\delta,\lambda_1)$,
\begin{align*}
	|u(z)|\leqslant C(n) \cdot C_0\frac{z^{\delta+1}}{\lambda}.
\end{align*}

For the derivative $u'$ we can do the same computation as in \cite{SZ} to estimate $\mathcal{D}'(z)$ and $\mathcal{G}'(z)$. In fact, we have the following estimate:
\begin{lemma}\label{estimate of I' and K'}
   \begin{align*}
       \frac{1}{C}\cdot \frac{e^{y}}{\sqrt{y}}<I_{\frac{1}{n}}'(y)<C\cdot \frac{e^{y}}{\sqrt{y}}, \quad \frac{1}{C}\cdot \frac{e^{-y}}{\sqrt{y}}<-K_{\frac{1}{n}}'(y)<C\cdot \frac{e^{-y}}{\sqrt{y}}
   \end{align*} for some fixed constant $C$ and any $y>1$.
\end{lemma}
\begin{proof}
    Notice that
\begin{align*}
	I_{\frac{1}{n}}'(y) &= \frac{1}{2\pi}\int_0^\pi e^{y\cos \theta} \left(\cos\tfrac{(n+1)\theta}{n}+\cos\tfrac{(n-1)\theta}{n}\right)d\theta+\frac{\sin{\frac{\pi}{n}}}{2\pi}\int_0^\infty e^{-y\cosh t}\left(e^{-\frac{(n+1)t}{n}}+e^{\frac{(n-1)t}{n}}\right)dt.
\end{align*}
As in the proof of \cite{SZ} Prop. 3.3, we know that for any $\nu \in \RR $\begin{align*}
    &\int_0^{\infty} e^{-y \cosh t-\nu t} d t \leqslant e^{-y} \int_0^{\infty} e^{-\frac{y t^2}{2}-\nu t} d t \leqslant C(\nu) \cdot \frac{e^{-y}}{\sqrt{y}},\\
    &\left| \int_0^\pi e^{y \cos \theta} \cos (\nu \theta) d \theta\right| \leqslant e^y \int_0^{\frac{\pi}{3}} e^{-\frac{y \cdot \theta^2}{4}} d \theta+\frac{2 e^{\frac{y}{2}}}{3} \leqslant \frac{2 e^y}{\sqrt{\pi} \cdot \sqrt{y}}+\frac{2 e^{\frac{y}{2}}}{3} \leqslant \frac{10 e^y}{\sqrt{y}}.
\end{align*}
For $\nu > -1$ and $\nu \neq 0$, let $\eta_\nu=\min \left(\pi, \frac{\pi}{3|\nu|}\right)$, 
\begin{align*}
	&\quad\int_0^\pi e^{y \cos \theta} \cos (\nu \theta) d \theta =\int_0^{\eta_\nu} e^{y \cos \theta} \cos (\nu \theta) d \theta+\int_{\eta_\nu}^\pi e^{y \cos \theta} \cos (\nu \theta) d \theta\\
	&\geqslant \frac{1}{2} e^y \int_0^{\eta_\nu} e^{-\frac{\theta^2}{2} y} d \theta - \left|\int_{\eta_\nu}^\pi e^{y \cos \theta} \cos (\nu \theta) d \theta\right| \geqslant C(\nu) \frac{e^y}{\sqrt{y}} -\int_{\eta_\nu}^\pi e^{y \cos \theta} d \theta \\
 &\geqslant C(\nu) \frac{e^y}{\sqrt{y}}-\left(\pi-\eta_\nu\right) e^{\cos \left(\eta_\nu\right) y} \geqslant C(\nu) \frac{e^y}{\sqrt{y}}.
\end{align*}
So we get that $
     \frac{1}{C}\cdot \frac{e^{y}}{\sqrt{y}}<I_{\frac{1}{n}}'(y)<C\cdot \frac{e^{y}}{\sqrt{y}}$.

On the other hand,
\begin{align*}
	K_{\frac{1}{n}}'(y) &= -\frac{1}{2}\int_0^\infty e^{-y\cosh t}\left(\cosh{\frac{(n+1)t}{n}}+\cosh{\frac{(n-1)t}{n}}\right)dt\\
    &=-\frac{1}{2}(K_{\frac{n+1}{n}}(y)+K_{\frac{n-1}{n}}(y)).
\end{align*}
By the estimate of $K_\nu$, we know that $ \frac{1}{C}\cdot \frac{e^{-y}}{\sqrt{y}}<-K_{\frac{1}{n}}'(y)<C\cdot \frac{e^{-y}}{\sqrt{y}}$.
\end{proof}

\begin{corollary}
For $z > \sqrt[\leftroot{-3}\uproot{3}n]{\frac{n}{4\lambda_1}}$, there exists a constant $C$ which only depends on $n$ such that 
	\begin{align*}
		\frac{1}{C}\cdot \lambda^{\frac{1}{4}} z^{\frac{n-2}{4}}
		e^{-\frac{2}{\sqrt{n}}\lambda^{\frac{1}{2}}\cdot z^{\frac{n}{2}}}
		<-&\mathcal{D}'(z)<C\cdot \lambda^{\frac{1}{4}} z^{\frac{n-2}{4}}
		e^{-\frac{2}{\sqrt{n}}\lambda^{\frac{1}{2}}\cdot z^{\frac{n}{2}}},\\
        \frac{1}{C}\cdot \lambda^{\frac{1}{4}} z^{\frac{n-2}{4}}
		e^{ \frac{2}{\sqrt{n}}\lambda^{\frac{1}{2}}\cdot z^{\frac{n}{2}}}
        <\quad&\mathcal{G}'(z)<C\cdot \lambda^{\frac{1}{4}} z^{\frac{n-2}{4}}
		e^{ \frac{2}{\sqrt{n}}\lambda^{\frac{1}{2}}\cdot z^{\frac{n}{2}}}.
    \end{align*}   
\end{corollary}
\begin{proof}
This can be seen directly from computing $\mathcal{D}'$ and $\mathcal{G}'$ with the substitution in (\ref{expression D and G 0}).
    \begin{align*}
-\mathcal{D}'(z)&=-\frac{1}{2\sqrt{z}}K_{\frac{1}{n}}\left(2\sqrt{\tfrac{\lambda}{n}}\cdot z^{\frac{n}{2}}\right)-\sqrt{n\lambda}z^{\frac{n-1}{2}}K'_{\frac{1}{n}}\left(2\sqrt{\tfrac{\lambda}{n}}\cdot z^{\frac{n}{2}}\right),\\
\mathcal{G}'(z)&=\frac{1}{2\sqrt{z}}I_{\frac{1}{n}}\left(2\sqrt{\tfrac{\lambda}{n}}\cdot z^{\frac{n}{2}}\right)+\sqrt{n\lambda}z^{\frac{n-1}{2}}I'_{\frac{1}{n}}\left(2\sqrt{\tfrac{\lambda}{n}}\cdot z^{\frac{n}{2}}\right).
\end{align*}
By Lemma \ref{estimate of I' and K'} we get the estimate.
\end{proof}

 Consequently, by integration by parts as before we have the $C^1$ estimate of $u$:
 \begin{align*}
     |u'(z)| &= \left|2\left(\mathcal{D}'(z)\int_1^z\mathcal{G}(s)s^{n-1}v(s)ds+\mathcal{G}'(z)\int_z^\infty \mathcal{D}(s)s^{n-1}v(s)ds\right)\right|\leqslant C(n)\cdot C_0 \frac{z^{\delta+\frac{n}{2}}}{\lambda^{\frac{1}{2}}}
 \end{align*}
 and 
 \begin{align*}
     u''(z) = n\lambda z^{n-2} u +nz^{n-1}v\leqslant C(n) \cdot C_0 z^{\delta+n-1}
 \end{align*}
for $z>C(n,\delta,\lambda_1)$. \\
In the end, we get if $|v(z)|\leqslant C_0z^\delta$ on  $z>C(n,\delta, M)$, then for any $z>C(n,\delta, M)$ $$\left|u(z)\right|\leqslant C(n) \cdot C_0 \frac{z^{\delta+1}}{\lambda}, \quad \left|u'(z)\right|\leqslant C(n) \cdot C_0 \frac{z^{\delta+\frac{n}{2}}}{\lambda^{\frac{1}{2}}}, \quad \left|u''(z)\right|\leqslant C(n) \cdot C_0 z^{\delta+n-1}.$$
\end{proof}

\subsection{fundamental solution of non-zero mode}

In this section we focus on the non-zero mode: the equation 
\begin{align}\label{equation of nonzero mode}
    u''-(\frac{j^2n^2}{4}+n\lambda)z^{n-2} u = nz^{n-1}v.
\end{align}
By \cite{SZ} we have the decay solution $\mathcal{D}(z)$ and growth solution $\mathcal{G}(z)$ of the homogeneous equation $u''(z) = (\frac{j^2n^2}{4}+n\lambda)z^{n-2} u $ given by
\begin{align}\label{expression D and G 1}
	&\mathcal{D}(z)=e^{\frac{j z^n}{2}} \cdot \Psi^\flat\left(\beta, \alpha,-j z^n\right),\\
    &\mathcal{G}(z)=e^{\frac{j z^n}{2}} \cdot \Phi^{\sharp}\left(\beta, \alpha,-j z^n\right), 
\end{align}
where $\alpha = 1-\frac{1}{n}$, $\beta = \frac{n-1}{2n}-\frac{\lambda}{nj}\leqslant 0$, $\Phi^{\sharp}\left(\beta, \alpha,-j z^n\right)$ and $\Psi^\flat\left(\beta, \alpha,-j z^n\right)$ have the following expression:
	\begin{align}\label{expression phi and psi 1}
	    \Psi^\flat(\beta, \alpha, y) &= \frac{e^y}{\Gamma(\alpha-\beta)} \int_0^{\infty} e^{y s} s^{\alpha-\beta-1}(1+s)^{\beta-1} d s,\\
        \Phi^{\sharp}(\beta, \alpha, y) &=\frac{\Gamma(\alpha)}{\Gamma(\alpha-\beta)} \cdot e^y(-y)^{\beta - \alpha} \cdot \int_0^{\infty} e^{\frac{s}{y}} \cdot s^{\frac{\alpha-1}{2}-\beta} \cdot I_{\alpha-1}(2 \sqrt{s}) ds.
	\end{align}
 Let $Q = \alpha -\beta-1$, $\gamma_n = \frac{1}{2} +\frac{1}{n}$. Denote 
$$F(t) =y t+Q \log \frac{t}{t+1}.$$ Then $F$ is strictly concave in $\RR$ if $Q>0$. Let $t_0$ be the only critical point of $F$. We have $$t_0 =\frac{1}{2}\left(-1+\sqrt{1+\frac{4 Q}{-y}}\right).$$ Denote $$G(u) =-u^2+2(-y)^{\frac{1}{2}} \cdot u+\left(2 Q+\gamma_n\right) \log u.
$$ Then $G$ is strictly concave in $\RR_+$. Let $u_0$ be the only critical point of $G$. Then $$u_0  =\frac{(-y)^{\frac{1}{2}}}{2} \cdot\left(1+\sqrt{1+\frac{4 Q}{-y}+\frac{2 \gamma_n}{-y}}\right).$$
In \cite{SZ} by Laplace method, we can show the following estimate:
\begin{lemma}\cite{SZ}\label{estimate of psi and phi}
There is a constant $C$ which only depends on $n$ such that
\\when $Q\geqslant 1$, 
         \begin{align*}
C_n^{-1} \cdot Q^{-\frac{1}{4}-\frac{1}{2 n}} \cdot \frac{(-y)^{-1} \cdot e^{y+F\left(t_0\right)}}{\Gamma(Q+1)} \leqslant &\Psi^\flat(\beta, \alpha, y) \leqslant C_n \cdot Q^{\frac{1}{4}} \cdot \frac{e^{y+F\left(t_0\right)}}{\Gamma(Q+1)},\\
C_n^{-1} \cdot Q^{-\frac{1}{4}} \cdot \frac{(-y)^{\frac{2-n}{4n}} \cdot e^{y+G\left(u_0\right)}}{\Gamma(Q+1)} \leqslant &\Phi^{\sharp}(\beta, \alpha, y) \leqslant C_n \cdot \frac{(-y)^{\frac{2-n}{4n}} \cdot e^{y+G\left(u_0\right)}}{\Gamma(Q+1)};
\end{align*}
when $Q\leqslant 1$,
\begin{align*}
C_n^{-1} \cdot e^y \cdot(-y)^{\beta-\alpha} & \leqslant \Psi^\flat(\beta, \alpha, y) \leqslant e^y \cdot(-y)^{\beta-\alpha}, \\
C_n^{-1} \cdot(-y)^{-\beta} & \leqslant \Phi^{\sharp}(\beta, \alpha, y) \leqslant C_n \cdot(-y)^{-\beta}
\end{align*}for any $y\leqslant -1$.
\end{lemma}

\begin{corollary}\cite{SZ}  
There is a constant $C$ which only depends on $n$ such that
\\when $Q\geqslant 1$, 
           \begin{align*}
 C^{-1} \cdot \frac{Q^{-\frac{1}{4}-\frac{1}{2 n}}}{\Gamma(Q+1)} \cdot e^{-\frac{j z^n}{2}+F\left(t_0(z)\right)} \cdot\left(j z^n\right)^{-1} \leqslant \;&\mathcal{D}(z) \leqslant C \cdot \frac{Q^{\frac{1}{4}}}{\Gamma(Q+1)} \cdot e^{-\frac{j z^n}{2}+F\left(t_0(z)\right)}, \\
 C^{-1} \cdot Q^{-\frac{1}{4}} \cdot \frac{\left(j z^n\right)^{\frac{2-n}{4n}}}{\Gamma(Q+1)} \cdot e^{-\frac{jz^n}{2}+G\left(u_0(z)\right)} \leqslant \;&\mathcal{G}(z) \leqslant C \cdot \frac{\left(j z^n\right)^{\frac{2-n}{4n}}}{\Gamma(Q+1)} \cdot e^{-\frac{jz^n}{2}+G\left(u_0(z)\right)};
\end{align*}
when $Q\leqslant 1$,
            \begin{align*}
C^{-1} \cdot e^{-\frac{jz^n}{2}} \cdot\left(j z^n\right)^{\beta-\alpha} \leqslant \;&\mathcal{D}(z) \leqslant C \cdot e^{-\frac{jz^n}{2}} \cdot\left(jz^n\right)^{\beta-\alpha}, \\
C^{-1} \cdot e^{\frac{jz^n}{2}} \cdot\left(jz^n\right)^{-\beta} \leqslant \;&\mathcal{G}(z) \leqslant C \cdot e^{\frac{jz^n}{2}} \cdot\left(jz^n\right)^{-\beta}
\end{align*} for any $z>1$.
\end{corollary}

We also need the following lemma
\begin{lemma}\label{estimate of F and G}
For any $z\geqslant 1$, $e^{F\left(t_0(z)\right)+G\left(u_0(z)\right)} \leqslant C\cdot j^{\frac{n+2}{4n}}z^{\frac{n+2}{4}} e^{jz^n} e^{-Q} Q^{Q+\frac{n+2}{4n}}$.
\end{lemma}
\begin{proof}
    By similar straight forward computation as in \cite{SZ}.
\end{proof}
With those estimates, we can give a $C^0$ bound of $u(z)$. By \cite{SZ} we know that the Wronskian $$\mathcal{W}(\mathcal{G},\mathcal{D}) =\mathcal{G}(z)\mathcal{D}'(z)-\mathcal{G}'(z)\mathcal{D}(z) =\frac{\Gamma(\alpha-1)}{\Gamma(\alpha-\beta)}j^{\frac{1}{n}}.$$ Hence we have a solution of (\ref{equation of nonzero mode}) as follows:
\begin{align}
		u(z) &= \frac{\Gamma(\alpha-\beta)n}{\Gamma(\alpha-1)j^{\frac{1}{n}}}\left(\mathcal{D}(z)\int_1^z\mathcal{G}(s)s^{n-1}v(s)ds+\mathcal{G}(z)\int_z^\infty \mathcal{D}(s)s^{n-1}v(s)ds\right).
	\end{align}

Then we can have the following estimate of our solution:
\begin{proposition}\label{estimate of u wrt v 1}
	Recall that $\lambda_1$ is the first nonzero positive eigenvalue of $-\Delta_Y$.
	Let $v$ be a smooth function such that $|v(z)|\leqslant C_0 z^\delta$ for $z>1$. For any $\lambda$ such that $\lambda > \lambda_1>0$, we can find solution of equation $u''-(\frac{j^2n^2}{4}+n\lambda)z^{n-2} u = nz^{n-1}v$ such that $|u(z)|\leqslant C \cdot C_0\frac{z^{\delta+1}}{\frac{j^2n^2}{4}z^n+n\lambda}$, $|u''(z)|\leqslant C\cdot C_0z^{\delta +n-1}$ on $z>C$ for some constant $C>1$ only depends on $n$, $\lambda_1$ and $a$.
\end{proposition}
\begin{proof}
Similar as the zero-mode case, we estimate $\mathcal{D}(z)\int_1^z\mathcal{G}(s)s^{n-1}v(s)ds+\mathcal{G}(z)\int_z^\infty \mathcal{D}(s)s^{n-1}v(s)ds$.\\
By integration by parts we have
\begin{align*}
    &\mathcal{D}(z)\int_1^z\mathcal{G}(s)s^{\delta+n-1}ds + \mathcal{G}(z)\int_z^\infty \mathcal{D}(s)s^{\delta+n-1}ds\\
    =& -\mathcal{W}(\mathcal{G},\mathcal{D})\sum_{k=0}^{N-1} P(TP)^k(z^{\delta+n-1})+\mathcal{D}(z)\int_1^z\mathcal{G}(s)(TP)^N(s^{\delta+n-1})ds + \mathcal{G}(z)\int_z^\infty \mathcal{D}(s)(TP)^N(s^{\delta+n-1})ds,
\end{align*}
where $P, T:C^\infty(\RR_+)\to C^\infty(\RR_+)$ are given by 
\begin{align*}
    P(f) = \frac{f}{z^{n-2}(\frac{j^2n^2}{4}z^n+n\lambda)}, \quad T(f) = f''.
\end{align*}
By straight forward computation and induction we can see that $$(TP)^k(z^{\delta+n-1})\leqslant C(k,n,\delta) \frac{z^{\delta+n-1-2nk}}{j^{2k}}.$$

So by taking $N$ large enough we have $(TP)^N(z^{\delta+n-1})\leqslant C(n,\delta,j)z^{-M}$ for some $M>2$ which will be chosen later.\\
We first consider the case that $Q\geqslant 1$. 
By Lemma \ref{estimate of F and G} the first term becomes
\begin{align*}
   \mathcal{D}(z)\int_1^z\mathcal{G}(s)(TP)^N(s^{\delta+n-1})ds
   \leqslant \;&C \cdot e^{-\frac{j z^n}{2}+F\left(t_0(z)\right)}\int_1^z \left(j s^n\right)^{\frac{2-n}{4n}}s^{-M}\cdot e^{-\frac{js^n}{2}+G\left(u_0(s)\right)}ds\\
    \;\leqslant \;&C \cdot e^{-jz^n+G\left(u_0(z)\right)+F\left(t_0(z)\right)} \cdot  j^{\frac{2-n}{4n}} \cdot z^{\frac{2-n}{4}-M+1}\leqslant C(n,Q)\cdot j^{\frac{1}{n}} \cdot z^{2-M}.
\end{align*} 
For the second term, we have similar estimate:
\begin{align*}
    \mathcal{G}(z)\int_z^\infty \mathcal{D}(s)(TP)^N(s^{\delta+n-1})ds
    \leqslant \;&C \cdot \left(j z^n\right)^{\frac{2-n}{4n}}\cdot e^{-\frac{j z^n}{2}+G\left(t_0(z)\right)}\int_z^\infty  s^{-M} e^{-\frac{js^n}{2}+F\left(u_0(s)\right)}ds\\
    \;\leqslant \;& C \cdot  e^{-j z^n+G\left(u_0(z)\right)+F\left(t_0(z)\right)} \cdot  j^{\frac{2-n}{4n}} \cdot z^{\frac{2-n}{4}-M+1}\leqslant C(n,Q)\cdot j^{\frac{1}{n}} \cdot z^{2-M}.
\end{align*} 
Then we consider the case where $Q\leqslant 1$
\begin{align*}
   &\mathcal{D}(z)\int_1^z\mathcal{G}(s)(TP)^N(s^{\delta+n-1})ds\leqslant \;C \cdot e^{-\frac{jz^n}{2}} \cdot\left(jz^n\right)^{\beta-\alpha} \int_1^z s^{-M}e^{\frac{js^n}{2}} \cdot\left(js^n\right)^{-\beta}ds\;\leqslant \;C\cdot j^{-\alpha} z^{1-n\alpha-M},\\
    &\mathcal{G}(z)\int_z^\infty \mathcal{D}(s)(TP)^N(s^{\delta+n-1})ds\leqslant \;C\cdot e^{\frac{jz^n}{2}} \cdot\left(jz^n\right)^{-\beta}\int_z^\infty  s^{-M}  e^{-\frac{js^n}{2}} \cdot\left(js^n\right)^{\beta-\alpha} ds\;\leqslant \;C\cdot j^{-\alpha} z^{1-n\alpha-M}.
\end{align*} 
 
So we have the uniform estimate for $u$ that for any $z>C(n,\delta)$,
\begin{align*}
	|u(z)|\leqslant C \frac{z^{\delta+1}}{\frac{j^2n^2}{4}z^n+n\lambda}.
\end{align*}

In the end, we get if $|v(z)|\leqslant C_0z^\delta$ on  $z>C(n,\delta, M)$, then for any $z>C(n,\delta, M)$
$$\left|u(z)\right|\leqslant C(n) \cdot C_0 \frac{z^{\delta+1}}{\frac{j^2n^2}{4}z^n+n\lambda}, 
\quad \left|u'(z)\right|\leqslant C(n) \cdot C_0 z^{\delta+n}, 
\quad \left|u''(z)\right|\leqslant C(n) \cdot C_0 z^{\delta+n-1}.$$
\end{proof}
\begin{remark}
	Even though the separation of variable method is very explicit, we can only get a bound of $u$ with respect to the polynomial growth order  of $v$ rather than the function $v$ itself. This is mainly because the behavior of operator $T$ and $P$ is not clear for general function.
\end{remark}
\bibliographystyle{plain}
\bibliography{ref}

\begin{thebibliography}{10}

\bibitem{apostolov2023hamiltonian}
Vestislav Apostolov and Charles Cifarelli.
\newblock Hamiltonian $2$-forms and new explicit calabi--yau metrics and gradient steady k\"ahler--ricci solitons on $\mathbb{C}^n$, 2023.

\bibitem{collins2022complete}
Tristan~C. Collins and Yang Li.
\newblock Complete calabi-yau metrics in the complement of two divisors, 2022.

\bibitem{collins2024free}
Tristan~C. Collins, Freid Tong, and Shing-Tung Yau.
\newblock A free boundary monge-amp\`ere equation and applications to complete calabi-yau metrics, 2024.

\bibitem{collins2016singular}
Tristan~C. Collins and Valentino Tosatti.
\newblock A singular demailly–p\u{a}un theorem.
\newblock {\em Comptes Rendus Mathematique}, 354(1):91--95, 2016.

\bibitem{conlon2013asymptotically}
Ronan~J. Conlon and Hans-Joachim Hein.
\newblock {Asymptotically conical Calabi–Yau manifolds, I}.
\newblock {\em Duke Mathematical Journal}, 162(15):2855 -- 2902, 2013.

\bibitem{conlon2017new}
Ronan~J Conlon and Fr{\'e}d{\'e}ric Rochon.
\newblock New examples of complete calabi-yau metrics on $\mathbb{C}^n$ for $n\geqslant 3$.
\newblock {\em arXiv preprint arXiv:1705.08788}, 2017.

\bibitem{hsvz2}
Jeff~Viaclovsky Hans-Joachim~Hein, Song~Sun and Ruobing Zhang.
\newblock Asymptotically calabi metrics and weak fano manifolds.
\newblock 2023.

\bibitem{HeinThesis}
{Hans-Joachim} Hein.
\newblock On gravitational instantons.
\newblock 2010.

\bibitem{hein2012gravitational}
Hans-Joachim Hein.
\newblock Gravitational instantons from rational elliptic surfaces.
\newblock {\em Journal of the American Mathematical Society}, 25(2):355--393, 2012.

\bibitem{hsvz1}
{Hans-Joachim} Hein, Song Sun, Jeff Viaclovsky, and Ruobing Zhang.
\newblock Nilpotent structures and collapsing ricci-flat metrics on the k3 surface.
\newblock {\em Journal of the American Mathematical Society}, 35:123--209, 2022.

\bibitem{kohn1963harmonic}
J.~J. Kohn.
\newblock Harmonic integrals on strongly pseudo-convex manifolds, i.
\newblock {\em Annals of Mathematics}, 78(1):112--148, 1963.

\bibitem{li2019new}
Yang Li.
\newblock A new complete calabi--yau metric on $\mathbb{C}^3$.
\newblock {\em Inventiones mathematicae}, 217(1):1--34, 2019.

\bibitem{min2023construction}
Daheng Min.
\newblock Construction of higher dimensional alf calabi-yau metrics, 2023.

\bibitem{mori3fold}
Shigefumi Mori and Shigeru Mukai.
\newblock Classification of fano 3-folds with $b_2\geqslant 2$.
\newblock {\em manuscripta mathematica}, 36(2):147--162, 1981.

\bibitem{Ohsawa1984VanishingTO}
Takeo Ohsawa.
\newblock Vanishing theorems on complete k{\"a}hler manifolds.
\newblock {\em Publications of The Research Institute for Mathematical Sciences}, 20:21--38, 1984.

\bibitem{SZ}
Song Sun and Ruobing Zhang.
\newblock A liouville theorem on asymptotically calabi spaces.
\newblock {\em Calculus of Variations and Partial Differential Equations}, 60(3):103, 2021.

\bibitem{sun2022collapsing}
Song Sun and Ruobing Zhang.
\newblock Collapsing geometry of hyperk\"ahler 4-manifolds and applications, 2022.

\bibitem{szekelyhidi2020new}
G{\'a}bor Sz{\'e}kelyhidi.
\newblock {Degenerations of $\mathbb{C}^{n}$ and Calabi--Yau metrics}.
\newblock {\em Duke Mathematical Journal}, 168(14):2651 -- 2700, 2019.

\bibitem{TakayamaS2000Scow}
S~Takayama.
\newblock Simple connectedness of weak fano varieties.
\newblock {\em Journal of algebraic geometry}, 9(2):403--407, 2000.

\bibitem{tian1990complete}
Gang Tian and Shing-Tung Yau.
\newblock Complete k\"ahler manifolds with zero ricci curvature. i.
\newblock {\em Journal of the American Mathematical Society}, 3(3):579--609, 1990.

\bibitem{van2008construction}
Craig van Coevering.
\newblock A construction of complete ricci-flat k{\"a}hler manifolds.
\newblock {\em arXiv: Differential Geometry}, 2008.

\bibitem{yau1978ricci}
Shing-Tung Yau.
\newblock On the ricci curvature of a compact kähler manifold and the complex monge-ampére equation, i.
\newblock {\em Communications on Pure and Applied Mathematics}, 31(3):339--411, 1978.

\end{thebibliography}
\end{document}